\documentclass[reqno,11pt]{amsart}
 
\usepackage{amsmath}
\usepackage{amsthm,amssymb,color,comment}
\usepackage{bbm}
\usepackage{mathrsfs}
\usepackage{hyperref}
\usepackage[TS1,T1]{fontenc}
\usepackage[utf8]{inputenc}
\usepackage{dsfont}
\usepackage{tikz}
\usepackage{pgfplots}
\pgfplotsset{compat=1.18}
\usepackage{enumitem}
\usepackage{subfigure}
\usepackage{mathtools}
\usepackage{bbold}
\usepackage{esint}
\usepackage[sort]{cite}

\numberwithin{equation}{section}

\newtheorem{thm}{Theorem}[section]

\newtheorem{proposition}[thm]{Proposition}
\newtheorem{lem}[thm]{Lemma}

\newtheorem{Def}[thm]{Definition}

\theoremstyle{definition}

\newtheorem{rem}[thm]{Remark}
\newtheorem{cor}[thm]{Corollary}

\DeclareMathOperator*{\essinf}{ess\,inf}

\newcommand{\R}{\mathbb{R}}
\newcommand{\Td}{\mathbb{T}^{d}}
\newcommand{\Rd}{\mathbb{R}^{d}}

\newcommand{\eps}{\varepsilon}

\newcommand{\diff}{\mathop{}\!\mathrm{d}}

\newcommand{\ueps}{u_{\varepsilon}}
\newcommand{\veps}{v_{\varepsilon}}
\newcommand{\seps}{s_{\varepsilon}}

\newcommand{\weaks}{\overset{\ast}{\rightharpoonup}}

\newcommand{\weak}{\rightharpoonup}

\makeatletter
\newcommand{\doublewidetilde}[1]{{%
  \mathpalette\double@widetilde{#1}%
}}
\newcommand{\double@widetilde}[2]{%
  \sbox\z@{$\m@th#1\widetilde{#2}$}%
  \ht\z@=.9\ht\z@
  \widetilde{\box\z@}%
}
\makeatother

\author[J. A. Carrillo]{Jos{\'e} A. Carrillo}
\address{{\it Jos{\'e} A. Carrillo: } Mathematical Institute, University of Oxford, Woodstock Road, Oxford, OX2 6GG, United Kingdom}
\email{carrillo@maths.ox.ac.uk}

\author[A. J\"ungel]{Ansgar J{\"u}ngel}
\address{{\it Ansgar J{\"u}ngel: } Institute of Analysis and scientific Computing, TU Wien, Wiedner hauptstr. 8--10, 1040 Wien, Austria}
\email{juengel@tuwien.ac.at}

\author[J. Skrzeczkowski]{Jakub Skrzeczkowski}
\address{{\it Jakub Skrzeczkowski: } St John's College, University of Oxford, St Giles, Oxford, OX1 3JP, United Kingdom \& Mathematical Institute, University of Oxford, Woodstock Road, Oxford, OX2 6GG, United Kingdom}
\email{jakub.skrzeczkowski@maths.ox.ac.uk}

\author[Y. Yao]{Yao Yao}
\address{{\it Yao Yao: } Department of Mathematics, National University of Singapore, Block S17, 10 Lower Kent Ridge Road, Singapore 119076}
\email{yaoyao@nus.edu.sg}

\begin{document}

\title[Non-uniqueness of weak solutions to cross-diffusion systems]{Non-uniqueness of weak solutions to cross-diffusion systems with advection}

\begin{abstract}
A cross-diffusion system with advection is considered on the whole line, describing the dynamics of two segregating population species. Starting from two initial densities supported on the half-lines $x\leq 0$ and $x\geq 0$, respectively, we construct two distinct solutions of the system: one pair of densities remains confined to their initial supports and stay completely segregated, while the second pair of densities begins to invade the opposite half-line after a finite time (mixing). To the best of our knowledge, this provides one of the first examples of non-uniqueness for this class of equations, together with an explicit demonstration of mixing phenomena. The construction produces infinitely many mixing solutions and applies throughout the full range of pressure exponents. For a certain range of exponents, we additionally obtain quantitative estimates on the mixing process. 
\end{abstract}



\maketitle

\setcounter{tocdepth}{1}

\section{Introduction}

Spatial segregation phenomena in interacting biological populations have attracted considerable attention in recent decades. Among the mathematical models proposed to describe such processes, Busenberg--Travis systems \cite{BuTr83,MR736508}, and more generally aggregation-diffusion systems as in \cite{MR3783102,MR5028313} for instance, provide a particularly influential framework for investigating the emergence of segregated states through strong interspecific competition. These systems arise as cross-diffusion equations for competing species and exhibit a wide variety of qualitative behaviors, including coexistence, segregation, entropy structure, and the formation of sharp interfaces \cite{MR821681,MR3783102}. Cross-diffusion systems have attracted considerable attention over the past several years, both because of the substantial analytical challenges they pose \cite{MR4188329,MR5033040,MR3740386,MR4939528,MR4000848} and because of their broad relevance across a variety of scientific applications \cite{murakawa2015continuous,MR3948738,MR4902815}. Despite substantial progress in understanding the existence and asymptotic properties of solutions, fundamental questions like their uniqueness or non-uniqueness remain basically unsolved. In this work, we prove the non-uniqueness of weak solutions to cross-diffusion systems for two species with opposite advection generalizing Busenberg-Travis systems. We focus on cross-diffusion equations for two species of the form 
\begin{equation}\label{eq:cross-diffusion-intro-general-velocities}
\begin{split}
&\partial_t u = \partial_x(u\, \partial_x p(s)) + \partial_x(u \, \partial_x V^1),\\
&\partial_t v = \partial_x(v\, \partial_x p(s)) + \partial_x(v \, \partial_x V^2)
\quad\mbox{in }\R,\ t>0, \\
&u(0)=u^0,\quad v(0)=v^0\quad\mbox{in }\R. 
\end{split}
\end{equation}
Here, $u$ and $v$ represent nonnegative population densities, $s = u+v$ denotes the total density, $V^1$ and $V^2$ are (e.g.\ environmental) potentials, and  
\begin{equation}\label{eq:darcy_law_introduction}
p(s) = \begin{cases}
\frac{1}{\alpha - 1} s^{\alpha-1} &\mbox{ if } \alpha \neq 1,\\
\log s &\mbox{ if } \alpha = 1,
\end{cases} \qquad \qquad \mbox{ for } \alpha \in (0,\infty).
\end{equation}
System \eqref{eq:cross-diffusion-intro-general-velocities} corresponds to mass-balance equations with the velocity $v=-\partial_x p(s)$, which can be interpreted as the Darcy law, and accordingly $p(s)$ corresponds to the mixture pressure. Another interpretation comes from many-particle systems. Indeed, systems of the type \eqref{eq:cross-diffusion-intro-general-velocities} with/without linear/nonlinear advection are the mean-field-type limit of moderately interacting particle systems for an arbitrary number of species \cite{CDJ19,MR5033616,CHS25}. Furthermore, equations \eqref{eq:cross-diffusion-intro-general-velocities} emerge as the zero-inertia limit of compressible Navier--Stokes systems with relaxation and pressure force \cite{CCDJ25}.

The original works \cite{BuTr83,MR736508} without advection consider the ideal gas law $p(s)=s$, i.e.\ the case $\alpha=2$. The regime $\alpha>1$ is particularly significant, since the total density is then expected to exhibit finite-speed propagation, similarly to the solutions to the porous-medium equation, see \cite[Chap.~14]{MR2286292}. In addition, the incompressible limit $\alpha\to \infty$ reveals a close connection between aggregation-diffusion systems and free-boundary problems for tissue growth \cite{MR3162474,MR429164}.  

The existence of weak solutions to \eqref{eq:cross-diffusion-intro-general-velocities} in bounded intervals without advection was proved in \cite{MR887658} for $p(s)=s$ and in \cite{MR2652018} for general increasing functions $p$. The results were generalized to several space dimensions in \cite{BHIM12} for $p(s)=s$ and in \cite{DrJu20,MR4179253} for $p(s)=s^\alpha$ with $\alpha>1$. These results hold for non-segregated initial data $u^0+v^0>0$. The variational splitting scheme allows the authors of \cite{MR3870087} to prove conservation of segregation for initially segregated one-dimensional data even in the presence of vacuum. In \cite{MR4880211}, vacuum is allowed but an upper bound for the initial data is required. A general proof in one space dimension was recently given in \cite{skrzeczkowski2026global}; also see the stability result of \cite{Elb26}. 

The main analytical difficulty of \eqref{eq:cross-diffusion-intro-general-velocities} stems from its mixed hyperbolic–parabolic structure.
Indeed, the total density satisfies the porous-medium equation with quadratic nonlinearity, whose solution is smooth if the initial datum is positive. The relative densities $u/s$ and $v/s$ solve hyperbolic transport equations, whose solutions do not exhibit any regularizing effect. This structure is explored in a more general context in \cite{DHJ23} and analyzed for measure-valued solutions in \cite{HoJu25}. 

The additional advection terms $V^i:\R\to\R$ model spatial heterogeneity. They may also incorporate nonlocal interactions by taking $V^i = K^{i,1}\ast u + K^{i,2}\ast v$, where the kernels $K^{i,j}$ account for repulsive or attractive interactions between species \cite{MR2279324,MR3948738,CSS26}. The presence of advection gives rise to substantial mathematical difficulties, already at the level of proving the existence of weak solutions. Such a result has been obtained only recently for the full range $\alpha\in(0,\infty)$ in \cite{skrzeczkowski2026global}. We also refer to \cite{MR3795211} for a result in the case $\alpha>1$, where the velocities are chosen so that the two species do not mix, and to \cite{alpar_guy_newwork, elbar2025cross} for a complete treatment of the fast-diffusion regime $\alpha\leq 1$. Another recent result \cite{santambrogio2026segregated} establishes existence of segregated solutions (for segregated initial conditions) for $\alpha > 1/3$.

The question of uniqueness of solutions to \eqref{eq:cross-diffusion-intro-general-velocities}, even without advection, is very delicate due to the lack of parabolicity. If the initial total density $u^0+v^0$ is strictly positive, there is a unique classical solution \cite{DrJu20}, while the uniqueness of solutions cannot generally be expected if $u^0+v^0\ge 0$. The weak--strong uniqueness property for \eqref{eq:cross-diffusion-intro-general-velocities} can be proved as in \cite{ChJu19}, providing the stability of strong solutions within the class of weak solutions. It is shown in \cite{GSV15} that for every smooth function $\Phi(t)$ such that $\Phi(0)=u^0+v^0$, there exists a new total density $s$ such that $s(t)=\Phi(t)$ on the interface between the two segregated species. To the best of our knowledge, we are not aware of any work that explicitly proves the non-uniqueness of the pair $(u,v)$ of weak solutions. In this paper, we establish this property.

More precisely, we prove, starting from segregated initial data, that there exist two solutions to the system
\begin{equation}\label{eq:intro-toy-model-to-study-nonuniqueness}
\begin{split}
\partial_t u &= \partial_x(u\, \partial_x p(s)) - \partial_x u,\\
\partial_t v &= \partial_x(v\, \partial_x p(s)) + \partial_x v\quad\mbox{in }\R,\ t>0, \\
u(0)&=u^0,\quad v(0)=v^0\quad\mbox{in }\R. 
\end{split}
\end{equation}
System \eqref{eq:intro-toy-model-to-study-nonuniqueness} is precisely \eqref{eq:cross-diffusion-intro-general-velocities} for $V^1(x) = -x$, $V^2(x) = x$. One solution will stay trivially segregated for all $t \geq 0$, while the second one will mix after some time~$t>0$. The weak solutions to \eqref{eq:intro-toy-model-to-study-nonuniqueness} are defined as follows. We interpret the terms $u\, \partial_x p(s)$, $v\, \partial_x p(s)$ as 
\begin{equation}\label{eq:pressure_term_rewritten_derivative_intro}
u\, \partial_x p(s) = u \, p'(s) \, \partial_x s = u\, s^{\alpha-2} \, \partial_x s = \frac{u}{s} \,  s^{\alpha-1} \, \partial_x s = \frac{1}{\alpha} \, \frac{u}{s} \, \partial_x s^{\alpha},
\end{equation}
which allows us to nicely handle the pressure term for $\alpha \leq 1$. We always define $u/s= 0$ on the set $\{s = 0\}$. In fact, this does not matter because we have $\partial_x s^{\alpha} = 0$ a.e. on the set $\{s = 0\}$, see for instance \cite[p.~314]{MR2759829}.

\begin{Def}\label{def:weak_sol}
We say that a pair $(u,v)$ of two nonnegative functions $u, v \in L^{\infty}(0,T; L^1(\R))$ with $s:=u+v$ is a {\em weak solution} to \eqref{eq:intro-toy-model-to-study-nonuniqueness} with initial conditions $u^0, v^0 \in L^1(\R)$ if $\partial_x s^{\alpha} \in L^2((0,T)\times\R)$ and for all $\varphi, \phi \in C_c^{\infty}([0,T)\times\R)$ we have
\begin{equation}\label{eq:weak_form_original_form_u}
\int_0^T \int_{\R} u\, \partial_t \varphi \diff x \diff t + \int_{\R} u^0(x) \, \varphi(0,x) \diff x = \int_0^T \int_{\R} \left(\frac{1}{\alpha} \, \frac{u}{s}\, \partial_x s^{\alpha} - u \right)\, \partial_x \varphi \diff x \diff t,
\end{equation}
\begin{equation}\label{eq:weak_form_original_form_v}
\int_0^T \int_{\R} v\, \partial_t \phi \diff x \diff t + \int_{\R} v^0(x) \, \phi(0,x) \diff x = \int_0^T \int_{\R} \left(\frac{1}{\alpha} \, \frac{v}{s}\, \partial_x s^{\alpha} + v\right)\, \partial_x \phi \diff x \diff t.
\end{equation}
\end{Def}

\begin{Def}\label{def:regular_weak_sln}
We say that the weak solution $(u, v)$ to \eqref{eq:intro-toy-model-to-study-nonuniqueness} with the initial conditions $u^0, v^0 \in L^1(\R) \cap L^{\infty}(\R)$, $u^0\,|x|, v^0\,|x| \in L^1(\R)$ is a {\em regular weak solution} if it satisfies additionally $u, v \in L^{\infty}(0,T; L^1(\R)) \cap L^{\infty}(0,T; L^\infty(\R))$, $\partial_x s^{\frac{\alpha}{2}} \in L^2((0,T)\times\R)$, as well as $u\,|x|, v\,|x| \in L^{\infty}(0,T; L^1(\Rd))$.
\end{Def}

The additional regularity of regular weak solutions implies that the fluxes $\frac{1}{\alpha}\frac{u}{s}\, \partial_x s^{\alpha} - u$ and $\frac{1}{\alpha} \, \frac{v}{s}\, \partial_x s^{\alpha} + v$ belong to $L^2((0,T)\times\R)$, so by considering in \eqref{eq:weak_form_original_form_u}--\eqref{eq:weak_form_original_form_v} test functions $\varphi, \phi \in C_c^{\infty}((0,T)\times\R)$ which are dense in $L^2(0,T; H^1(\R))$, we immediately deduce an equivalent definition of weak solutions, which will be exploited in this paper.

\begin{lem}\label{lem_additional_prop_reg_weak_sln}
Let $(u,v)$ be a regular weak solution to \eqref{eq:intro-toy-model-to-study-nonuniqueness}. Then, $\partial_t u, \partial_t v \in L^2(0,T; H^{-1}(\Rd))$ and for all $\varphi, \phi \in L^2(0,T; H^1(\Rd))$ we have
\begin{equation}\label{eq:weak_form_H-1_u}
    \int_0^T \langle \partial_t u, \varphi \rangle_{H^{-1}(\R), H^1(\R)} \diff t = - \int_0^T \int_{\R} \left(\frac{1}{\alpha} \, \frac{u}{s}\, \partial_x s^{\alpha} - u \right)\, \partial_x \varphi \diff x \diff t,
    \end{equation}
    \begin{equation}\label{eq:weak_form_H-1_v}
    \int_0^T \langle \partial_t v, \phi \rangle_{H^{-1}(\R), H^1(\R)} \diff t = - \int_0^T \int_{\R} \left(\frac{1}{\alpha} \, \frac{v}{s}\, \partial_x s^{\alpha} + v \right)\, \partial_x \phi \diff x \diff t.
    \end{equation}
Moreover, $u, v \in C([0,T]; H^{-1}(\R))$ and $\lim_{t\to 0} u(t,\cdot) = u^0(\cdot)$, $\lim_{t\to 0} v(t,\cdot) = v^0(\cdot)$ in~$H^{-1}(\R)$.
\end{lem} 

To see the last part of Lemma~\ref{lem_additional_prop_reg_weak_sln}, we notice that \eqref{eq:weak_form_original_form_u}--\eqref{eq:weak_form_original_form_v} are valid (by density) for test functions of the form $\varphi(t,x) = \eta(t)\,\psi(x)$ with $\eta\in C[0,T]$, $\eta'\in L^{\infty}(0,T)$, $\eta$ supported on $[0,T)$, and $\psi \in H^1(\R)$. For fixed $t<T$ and $\delta>0$ such that $t+\delta <T$, we consider $\eta_{\delta}$ such that $\eta_{\delta}(\tau) = 1$ for $\tau \in [0,t]$, $\eta_{\delta}(\tau)=0$ for $\tau\in[t+\delta,\infty)$, and $\eta_{\delta}(\tau)=1-(\tau-t)/\delta$ for $\tau \in [t, t+\delta]$. Hence, $\eta_{\delta} \to \mathds{1}_{[0,t]}$ as $\delta \to 0$ and $\eta'_{\delta} = -(1/\delta)\,\mathds{1}_{[t,t+\delta]}$. Therefore,
$$
\int_0^T \int_{\R} u\, \partial_t \eta_{\delta} \, \psi \diff x \diff t = -\frac{1}{\delta}\int_{t}^{t+\delta} \langle u(\tau, \cdot),  \psi \rangle_{H^{-1}(\R), H^1(\R)}  \diff \tau \to - \langle u(t, \cdot),  \psi \rangle_{H^{-1}(\R), H^1(\R)}
$$
by $u \in C([0,T]; H^{-1}(\R))$. It follows from \eqref{eq:weak_form_original_form_u} that 
\begin{equation}\label{eq:weak_formulation_H^-1,H^1_pointwise_in_time}
- \langle u(t, \cdot),  \psi \rangle_{H^{-1}(\R), H^1(\R)} + \langle u^0(\cdot),  \psi \rangle_{H^{-1}(\R), H^1(\R)}= \int_0^t \int_{\R} \left(\frac{1}{\alpha} \, \frac{u}{s}\, \partial_x s^{\alpha} - u \right)\, \partial_x \psi \diff x \diff \tau,
\end{equation}
which implies that $\lim_{t\to 0} u(t,\cdot) = u^0(\cdot)$ in $H^{-1}(\R)$, since $\frac{1}{\alpha} \, \frac{u}{s}\, \partial_x s^{\alpha} - u \in L^2((0,T)\times\R)$. The same argument works for $v$.

Our main result reads as follows:

\begin{thm}\label{thm:main_nonuniqueness}
Let $\alpha >0$. Let the initial condition $u^0 \geq 0$, $u^0 \neq 0$, $u^0 \in L^1(\R)\cap L^{\infty}(\R)$, $u^0$~be compactly supported on $(-\infty,0]$, and let $v^0(x) = u^0(-x)$. There exist two pairs $(u_{A}, v_{A})$, $(u_{B}, v_{B})$ of regular weak solutions to \eqref{eq:intro-toy-model-to-study-nonuniqueness} in the sense of Definition \ref{def:regular_weak_sln} such that:
\begin{itemize}
\item $(u_{A}, v_{A})$ stays segregated, i.e.\ for a.e. $t$, $\int_{0}^{\infty} u_{A}(t,x) \diff x = \int_{-\infty}^0 v_{A}(t,x)\diff x =~0$,
\item $(u_{B}, v_{B})$ mixes in finite time, i.e., there exists a set of times $\mathcal{F}$ of positive measure such that for $t\in\mathcal{F}$, we have $\int_{0}^{\infty} u_B(t,x) \diff x\!>\!0$. Moreover, 
\begin{equation}\label{eq:estimate_time_mixing_main_statement}
\essinf \mathcal{F} \leq -\frac{\int_{-\infty}^0 u^0(x) \, x \diff x}{\int_{-\infty}^0 u^0(x) \diff x},
\end{equation}
where $\essinf D = \sup\{d\in \R: D \cap (-\infty,d] \mbox{ is a null set}\}$ for any set $D\subset \R$.
\end{itemize}
\end{thm}

To the best of our knowledge, this is the first result establishing non-uniqueness of weak solutions for cross-diffusion systems with advection. We note, however, that an independent construction by Charles Elbar and Guy Parker appeared online at the same time as the present work \cite{Elbar_Parker}.

A simple consequence is that one can construct two (and, in fact, infinitely many) mixing solutions by gluing in time a non-mixing solution with a mixing one, leading to the following corollary.

\begin{cor}\label{cor:nonuniqueness_mixing}
Under the assumptions of Theorem \ref{thm:main_nonuniqueness}, there exist infinitely many distinct regular weak solutions to \eqref{eq:intro-toy-model-to-study-nonuniqueness} that mix in finite time. In contrast, there exists a unique regular weak solution $(u,v)$ that remains segregated on $(-\infty,0]$ and $[0,\infty)$, in the sense that $\int_{0}^{\infty} u(t,x) \diff x = \int_{-\infty}^0 v(t,x)\diff x =~0$ for a.e. $t\geq 0$.
\end{cor}

Finding a selection principle for the constructed solutions is a challenging open problem.
We can also prove some quantitative bounds as outlined below. 
\begin{thm}\label{thm:quantitative_bounds}
Let $\alpha \in (1,2)$. Let $(u,v)$ be the mixing solution from Theorem \ref{thm:main_nonuniqueness} such that $u(t,x)=v(t,-x)$ a.e. and let $s=u+v$. Then, there exists constants $C_1(\alpha, s^0)$, $C_2(\alpha, s^0)$, $C_3(\alpha, s^0)$, $C_4(\alpha, s^0)$ depending on $\alpha$, $\|s^0\|_{L^1(\R)}$, $\|s^0\|_{L^{\infty}(\R)}$, $\int_{\R} s^0\, |x|^2 \diff x$ such that for a.e.\ $t$ at least one of the following holds
\begin{equation}\label{eq:quant_bound_main_thm_log_one}
- C_4(\alpha,s^0) -C_1(\alpha, s^0)\, (1+\log (t+1)) + C_3(\alpha,s^0) \, t^{(\alpha+1)/(3\alpha)} \leq 2 \int_0^t \int_{-\infty}^0 {v^{2-\alpha}} \diff x \diff \tau,
\end{equation}
\begin{equation}\label{eq:quant_bound_main_thm_energy_one}
C_2(\alpha,s^0)\, (1+t) \leq  2 \int_{-\infty}^0 v(t,x)\, |x| \diff x + 2\,\alpha^2 \int_0^t \int_{-\infty}^0 v(\tau,x) \diff x \diff \tau.
\end{equation}
\end{thm}

\begin{rem} Some precisions are needed.
    \begin{itemize}
        \item [i)] For sufficiently large times, both left-hand sides of \eqref{eq:quant_bound_main_thm_log_one}--\eqref{eq:quant_bound_main_thm_energy_one} are positive and are growing with some power-law rate in $t$. Denoting their right-hand sides by $R_1(t)$ and $R_2(t)$, respectively, we infer the estimate $\sup\{R_1(t), R_2(t)\} \gtrsim t^{(\alpha+1)/3\alpha}$ for sufficiently large $t$, which gives a quantitative measurement on mixing. We note that all constants are fully computable, and we have kept the same notation as in the proof to facilitate the corresponding calculations for interested readers. For small times, the first lower bound may be negative, while the second lower bound may fail to hold; therefore, it does not rule out segregation at short times.
        \item[ii)] The restriction $\alpha < 2$ does not seem easy to overcome. On the other hand, it should not be difficult to extend the proof to the regime $\alpha > 1/3$ by exploiting the fact that the energy $\int_{\R} s^{\alpha}\diff x$ is then controlled by the second moment $\int_{\R} s|x|^2 \diff x$. From our perspective, the regime $\alpha > 1$ is particularly interesting, since the densities are expected to propagate with finite speed, and establishing mixing estimates should therefore be more challenging.
        \item[iii)] The integrability of $v^{2-\alpha}$ follows from Lemma \ref{lem:arbitrary_high_moments} together with the compact support of the initial condition.
        \item[iv)] The symmetry condition $u(t,x)=v(t,-x)$ is expected to be satisfied for the mixing solution, since it is obtained as a limit of the regularized system \eqref{eq:general_cross_diffusion_intro_viscosity}. If the latter has a unique solution, then one can check that both $(\ueps(t,x), \veps(t,x))$ and $(\veps(t,-x), \ueps(t,-x))$ are solutions which, by uniqueness, implies that $\ueps(t,x)=\veps(t,-x)$. Passing to the limit, we infer that $u(t,x)=v(t,-x)$ a.e.
        \item[v)]
        Theorem~\ref{thm:quantitative_bounds} can also be applied to the mixing solutions in Corollary~\ref{cor:nonuniqueness_mixing}, which are constructed by gluing in time the non-mixing and mixing solutions from Theorem~\ref{thm:main_nonuniqueness}; however, the estimates in Theorem~\ref{thm:quantitative_bounds} are valid only for times after the gluing.
    \end{itemize}
\end{rem}

In the absence of advection, it is well known that, in one spatial dimension, sufficiently smooth solutions preserve segregation whenever the initial data are segregated \cite{MR2652018}. In particular, by passing to the limit in a family of smooth approximations, one obtains the existence of segregated weak solutions; see also \cite{MR3870087,MR4072681}. In higher dimensions, the existence of segregated weak solutions was established in \cite{MR4929608}. In contrast, our construction relies on the presence of advection, in the sense that the mixing solution is obtained as the limit of a~more regular approximation. More precisely, we consider the viscous regularization
\begin{equation}\label{eq:general_cross_diffusion_intro_viscosity}
\begin{split}
\partial_t \ueps &= \partial_x(\ueps\, \partial_x p(\seps)) - \partial_x \ueps  + \eps\, \partial^2_x \ueps,\\
\partial_t \veps &= \partial_x(\veps\, \partial_x p(\seps)) + \partial_x \veps  + \eps\, \partial^2_x \veps,
\end{split}
\end{equation}
with $\seps=\ueps+\veps$. In passing to the limit $\eps \to 0$, we will use the following result from \cite[Theorem 1.3]{skrzeczkowski2026global}:
\begin{thm}\label{thm:main}
Let $u^0$, $v^0$ be the initial data as in Theorem \ref{thm:main_nonuniqueness}.  There exists a regular weak solution $(u,v)$ to \eqref{eq:intro-toy-model-to-study-nonuniqueness} in the sense of Definition \ref{def:regular_weak_sln}. It is obtained as a limit of solutions to~\eqref{eq:general_cross_diffusion_intro_viscosity}. More precisely, there exists a subsequence (not relabelled) such that 
\begin{itemize}
\item $\ueps \to u$, $\veps \to v$ a.e., strongly in $L^p((0,T)\times\R)$ for all $p\in[1,\infty)$ as well as weakly$^*$ in $L^\infty((0,T)\times\R)$,
\item $\partial_x \seps^{{\alpha}} \to \partial_x s^{{\alpha}}$ strongly in $L^2((0,T)\times\R)$.
\end{itemize}
\end{thm}

The version in \cite{skrzeczkowski2026global}, formulated for general velocity fields as in \eqref{eq:cross-diffusion-intro-general-velocities}, gives the strong convergences $\ueps \to u$, $\veps \to v$ only at the points where $\partial_x V^1 \neq \partial_x V^2$ which, in our case, coincides with $(0,T)\times\R$. Furthermore, the notion of regular weak solution is not used there but the regularity stated in \cite[Theorem 1.3]{skrzeczkowski2026global} coincides with the one required by Definition~\ref{def:regular_weak_sln}.

Let us give an idea of the proof. The non-mixing solution is constructed by solving the problem on two half-lines with Neumann boundary conditions, which prevent the solutions from leaving their respective half-lines; see Proposition \ref{prop:existence_non_mixing_sln}. The mixing solution is obtained through Theorem \ref{thm:main} as a limit of the viscous approximation. To prove that this solution indeed exhibits mixing, we study the quantity $\int_{\R} (u\log u + v\log v - s\log s) \diff x$. It is well-known that both $\int_{\R} (u \log u + v \log v) \diff x$ and $\int_{\R} s \log s \diff x$ are useful in the analysis of cross-diffusion systems \cite{elbar2025cross,skrzeczkowski2026global}. In Proposition~\ref{prop:trace_at_0_for_separated_solutions}, we prove that under the assumption that the limiting solution $(u,v)$ to \eqref{eq:general_cross_diffusion_intro_viscosity} stays segregated up to time $t$, we have
\begin{equation}\label{eq:intro_main_identity_to_get_contradiction}
\int_{\R} \big(u(t,x)\log u(t,x) + v(t,x)\log v(t,x) - s(t,x)\log s(t,x)\big) \diff x + 2 \int_0^t s(\tau,0) \diff \tau \leq 0.
\end{equation}
Hence, if $u$ and $v$ remain segregated up to time $t$, then $s(\tau,0) = 0$ for a.e.\ $\tau \in [0,t]$. We deduce from the identity
$$
\int_{-\infty}^0 u(t,x) \, x \diff x - \int_{-\infty}^0 u^0(x) \, x \diff x  = - \frac{1}{\alpha} \int_0^t s^{\alpha}(\tau,0) \diff \tau   + t\, \int_{-\infty}^0 u^0(x) \diff x,
$$
which holds whenever $u$ and $v$ remain segregated on $[0,t]$, that $\int_{-\infty}^0 u(t,x) \, x \diff x $ eventually becomes positive, yielding a contradiction. 

Several additional difficulties arise in the proof. In particular, the available regularity is insufficient to justify \eqref{eq:intro_main_identity_to_get_contradiction} directly and suitable approximation arguments are therefore required. For instance, the last term in~\eqref{eq:intro_main_identity_to_get_contradiction} would be easier to obtain if $\partial_x s$ existed. Since this is not known in general, one must instead exploit the weaker regularity provided by $\partial_x s^\alpha$ or $\partial_x s^{\frac{\alpha}{2}}$. We refer to Remark~\ref{rem:difficulty_if_arguing_when_s_smooth} for an explanation of why one cannot simply perform the computations at the level of $s_{\eps}$, which enjoys additional regularity, and then pass to the limit $\eps \to 0$. Moreover, we work on~$\R$, which requires some additional effort to control the behaviour at large arguments (one could instead work on $\Td$ but this would make the construction of the segregated solution more difficult). Finally, we prove that both solutions possess the same regularity, namely that they are regular weak solutions.

Last but not least, regarding the question if the mixing solution is indeed the preferred one among weak solutions, our work shows that the vanishing-viscosity limit from \cite{skrzeczkowski2026global} selects it. However, by Corollary~\ref{cor:nonuniqueness_mixing}, these solutions are not unique. An interesting open question in this direction is whether an additional condition can be found that selects the unique solution. A related issue arises in the context of the isentropic Euler equations of gas dynamics, where uniqueness and non-uniqueness results for various solution concepts have been shown \cite{GiKw22}.

\section{Proofs of Theorem \ref{thm:main_nonuniqueness}, Corollary~\ref{cor:nonuniqueness_mixing}, and Theorem~\ref{thm:quantitative_bounds}}

We first construct the solution that does not mix. 
\begin{proposition}\label{prop:existence_non_mixing_sln}
There exists a weak solution $(u_{A}, v_{A})$ to \eqref{eq:intro-toy-model-to-study-nonuniqueness} that stays segregated as stated in Theorem \ref{thm:main_nonuniqueness}.
\end{proposition}
\begin{proof}
Let $u$ be the unique solution to 
\begin{equation}\label{eq:Neumann_problem_on_a_half_space}
\partial_t u = \partial_x(u \, \partial_x p(u)) - \partial_x u = \frac{1}{\alpha} \partial^2_x u^{\alpha} - \partial_x u
\end{equation}
on $[0,T] \times (-\infty,0]$ with Neumann boundary condition $u \, \partial_x p(u) - u = 0$ at $x = 0$ and the initial condition $u^0$. Clearly, due to the boundary condition, the solution stays supported on $(-\infty, 0]$ (to see this, one can for instance compute $\partial_t \int_{-\infty}^0 u \diff x = 0$). We claim that $u$~has the following regularity
\begin{align}\label{eq:regularity_half_line_solution_u}
& u \in L^{\infty}(0,T; L^1(\R^-)) \cap L^{\infty}(0,T; L^{\infty}(\R^-)), \quad u\,|x| \in L^{\infty}(0,T; L^1(\R^{-})), \\
\label{eq:regularity_half_line_solution_u_2}
& \partial_x u^{\frac{\alpha}{2}}, \partial_x u^{\alpha} \in L^2((0,T)\times \R^-),
\end{align}
where $\R^- = (-\infty,0]$. This is very similar to \cite[Section~3.1]{skrzeczkowski2026global} but some changes are necessary due to the Neumann boundary condition. We provide a detailed computation in Appendix~\ref{app:estimates}.

Having established the regularity \eqref{eq:regularity_half_line_solution_u}--\eqref{eq:regularity_half_line_solution_u_2}, we introduce the function $v(x) = u(-x)$, which solves the PDE
$$
\partial_t v = \partial_x(v \, \partial_x p(v)) + \partial_x v = \frac{1}{\alpha} \partial^2_x v^{\alpha} + \partial_x v
$$
on $[0,T]\times [0,\infty)$ with initial condition $v^0(x) = u^0(-x)$ and the Neumann boundary condition. It follows that $v$ also satisfies \eqref{eq:regularity_half_line_solution_u}--\eqref{eq:regularity_half_line_solution_u_2} on $\R^+ := [0,\infty)$.

We now define $s=u+v$. We claim that $\partial_x s^{\alpha/2} = \partial_x u^{\alpha/2}\, \mathds{1}_{\R^-} + \partial_x v^{\alpha/2} \,\mathds{1}_{\R^+} \in L^2((0,T)\times\R)$ (and analogously for $\partial_x s^{\alpha}$). The integrability is clear from \eqref{eq:regularity_half_line_solution_u_2}. To see the differentiability, we compute for any $\varphi \in C_c^{\infty}((0,T)\times\R)$,
\begin{multline*}
\int_{0}^T \int_{\R} \partial_x \varphi \, s^{\alpha/2} \diff x \diff t = \int_{0}^T \int_{\R^-} \partial_x \varphi \, u^{\alpha/2} \diff x \diff t + \int_{0}^T \int_{\R^+} \partial_x \varphi \, v^{\alpha/2} \diff x \diff t \\
= - \int_{0}^T \int_{\R} \varphi\,(\partial_x u^{\alpha/2}\, \mathds{1}_{\R^-} + \partial_x v^{\alpha/2} \,\mathds{1}_{\R^+}) \diff x \diff t + \int_0^T \varphi(t,0)\,(u^{\alpha/2}(t,0) - v^{\alpha/2}(t,0)) \diff t.
\end{multline*}
The last term vanishes because the traces of $u^{\alpha/2}$ and $v^{\alpha/2}$ at $x=0$ have to be the same by construction ($u(t,x)=v(t,-x)$). It follows that the pair $u, v$ with $s=u+v$ has all the regularity required from a regular weak solution (Definition \ref{def:regular_weak_sln}).

Finally, we have to establish the weak formulation \eqref{eq:weak_form_original_form_u}--\eqref{eq:weak_form_original_form_v}. We note that the solution to \eqref{eq:Neumann_problem_on_a_half_space} with Neumann boundary condition satisfies
\begin{equation}\label{eq:weak_formulation_Neumann_problem_half_line}
\int_0^T \int_{\R^-} u\, \partial_t \varphi \diff x \diff t + \int_{\R^-} u^0(x) \, \varphi(0,x) \diff x = \int_0^T \int_{\R^-} \left(\frac{1}{\alpha} \, \partial_x u^{\alpha} - u \right)\, \partial_x \varphi \diff x \diff t,
\end{equation}
for all $\varphi \in C_c^{\infty}([0,T)\times (-\infty, 0])$. For $x\in\R^-$, we write $ \partial_x u^{\alpha} =  \frac{u}{s}\,\partial_x s^{\alpha}$ (recall that $\partial_x u^{\alpha}=0$ when $u=s=0$ by the estimate on $\partial_x u^{\frac{\alpha}{2}}$). We note that by extending $u=0$ for $x\in \R^+$, $\frac{u}{s}\,\partial_x s^{\alpha} = 0$ for $x \in \R^+$ (when $s>0$, $u/s = 0$ while when $s=0$, $u/s$ is defined to be zero). Hence, we can extend the identity \eqref{eq:weak_formulation_Neumann_problem_half_line} for $\varphi \in C_c^{\infty}([0,T)\times \R )$ concluding the proof of \eqref{eq:weak_form_original_form_u}. Analogous argument works for $v$ to derive \eqref{eq:weak_form_original_form_v}, concluding the proof.
\end{proof}

We now turn to the construction of the mixing solution. As a preliminary step in the contradiction argument, we show that if the limit of \eqref{eq:general_cross_diffusion_intro_viscosity} stays segregated, then $s(t,0) = 0$ for a.e.~$t$.  

\begin{rem}[Definition of $s(t,0)$]\label{rem:def_s_at_x_0} By Theorem \ref{thm:main}, the limit  pair $(u,v)$ is a weak solution. Hence, $\partial_x s^{\alpha} \in L^2((0,T)\times\R)$, which implies that the map $x\mapsto s(t,x)^{\alpha}$ is locally $\frac{1}{2}$-Hölder continuous for a.e.\ $t$. It follows that $x \mapsto s(t,x)$ is also continuous (note that $s(t,x)\geq 0$). Note however that $u$ and $v$ are defined on up to sets of measure zero, so the equality $s = u + v$ holds only a.e.\ as well. Consequently, $s(t,0)$ is understood as the value at $x=0$ of the unique continuous representative of $s(t,\cdot)$, rather than the pointwise sum $u(t,0)+v(t,0)$; see Figure~\ref{fig:sep_species} for an example.
\end{rem} 

\begin{figure}
\centering
\begin{tikzpicture}
\begin{axis}[
    axis lines=middle,
    axis line style={gray},
    tick style={gray},
    xmin=-1.6, xmax=1.6,
    ymin=-0.18, ymax=1.18,
    xtick={-1,0,1},
    ytick={0,1},
    width=7cm,
    height=5cm,
    clip=false,
    legend style={
        draw=none,
        fill=none,
        font=\small,
        at={(0.97,0.97)},
        anchor=north east
    }
]

\def\eps{0.005}

\addplot[thick,solid,forget plot] coordinates {(-1.6,-\eps) (-1,-\eps)};
\addplot[thick,solid,forget plot] coordinates {(-1,0) (0,1-\eps)};
\addplot[thick,solid,forget plot] coordinates {(0,-\eps) (1.6,-\eps)};

\addplot[thick,dashed,forget plot] coordinates {(-1.6,\eps) (0,\eps)};
\addplot[thick,dashed,forget plot] coordinates {(0,1+\eps) (1,0)};
\addplot[thick,dashed,forget plot] coordinates {(1,0) (1.6,\eps)};

\addlegendimage{thick,solid}
\addlegendentry{$u$}

\addlegendimage{thick,dashed}
\addlegendentry{$v$}

\end{axis}
\end{tikzpicture}
\caption{A segregated stationary solution for $\alpha = 2$ with initial conditions $u^0$, $v^0$ such that $\int_{\R} u^0(x) \diff x = \int_{\R} v^0(x) \diff x = \frac{1}{2}$. Both species have a jump at $x=0$. However, $s = u + v$ has the unique continuous representative and $s(t,0)$ is its value at $x=0$.}
\label{fig:sep_species}
\end{figure}
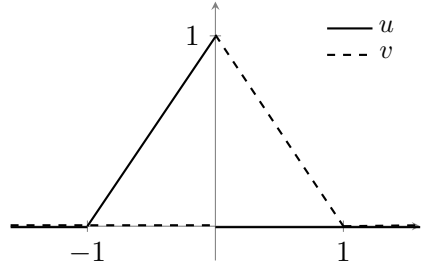

\begin{proposition}\label{prop:trace_at_0_for_separated_solutions}
Let $(u, v)$ be any limit of \eqref{eq:general_cross_diffusion_intro_viscosity} as $\eps \to 0$ given by Theorem \ref{thm:main} with initial conditions $u^0, v^0$ as in Theorem \ref{thm:main_nonuniqueness}. Let $s = u + v$. If $t^*$ is such that $\int_0^{\infty}u(t,x)\diff x = \int_{-\infty}^0 v(t,x)dx=0$ for a.e. $t \in [0,t^*]$ (i.e. if $u$ and $v$ stay segregated), then $s(t,0) = 0$ for a.e.\ $t \in [0,t^*]$, where $s(t,0)$ is defined in Remark \ref{rem:def_s_at_x_0}. 
\end{proposition}
\begin{proof} We divide the proof into three steps.

\underline{Step 1: Estimate on $\int_{\R} (u(t,x) \log u(t,x) + v(t,x) \log v(t,x)) \diff x$.} We claim that
\begin{equation}\label{eq:entropy_individual_inequality_the_limit}
\begin{aligned}
\int_{\R} &\big(u(t,x) \log u(t,x) + v(t,x) \log v(t,x)\big) \diff x + \frac{4}{\alpha^2} \int_0^t \int_{\R}  |\partial_x s^{\frac{\alpha}{2}}|^2 \diff x \diff \tau \\
&\leq \int_{\R} \big(u^0(x) \log u^0(x) + v^0(x) \log v^0(x)\big) \diff x.
\end{aligned}
\end{equation}
For the proof, we multiply the PDEs in \eqref{eq:general_cross_diffusion_intro_viscosity} by $\log \ueps$ and $\log \veps$ respectively, integrate by parts, and sum up to obtain
$$
\partial_t \int_{\R} \big(\ueps \log \ueps + \veps \log \veps\big) \diff x + \frac{4}{\alpha^2} \int_{\R} |\partial_x \seps^{\alpha/2}|^2 \diff x + \eps \int_{\R} \left(\frac{|\partial_x \ueps|^2}{\ueps} + \frac{|\partial_x \veps|^2}{\veps}\right) \diff x \leq 0.
$$
A rigorous justification of this inequality is standard and proceeds by introducing suitable regularizations together with lower semicontinuity arguments, relying on the fact that the viscous terms $\eps\, \partial^2_x \ueps$, $\eps\,\partial^2_x \veps$ guarantee the existence of $\partial_x \ueps$ and $\partial_x \veps$. We now integrate the inequality in time and disregard the viscosity terms:
\begin{equation}\label{eq:entropy_before_sending_eps_to_zero}
\begin{split}
\int_{\R} &\big(\ueps(t,x) \log \ueps(t,x) + \veps(t,x) \log \veps(t,x)\big) \diff x + \frac{4}{\alpha^2} \int_0^t \int_{\R} |\partial_x \seps^{\alpha/2}|^2 \diff x \diff \tau \\
&\leq \int_{\R} \big(u^0(x) \log u^0(x) + v^0(x) \log v^0(x)\big) \diff x.
\end{split}
\end{equation}
We want to pass to the limit $\eps \to 0$ in \eqref{eq:entropy_before_sending_eps_to_zero}. It is a standard calculus inequality (see, e.g. \cite[Lemma D.2]{carrillo2024stein}) that for any $\rho:\R \to [0,\infty)$,
\begin{equation}\label{eq:algebraic_inequalities_on_rho_log_rho}
\rho\log \rho + \rho\, |x| + \frac{2}{e} \, e^{-|x|/2} \geq 0.
\end{equation}
Moreover, we know that $\{\seps\,|x|\}_{\eps\in(0,1)}$ is uniformly bounded in $L^{\infty}(0,T; L^1(\R))$ \cite[Lemma 3.3]{skrzeczkowski2026global}. Proceeding to the limit $\eps\to0$, \eqref{eq:entropy_before_sending_eps_to_zero} and \eqref{eq:algebraic_inequalities_on_rho_log_rho} imply that $\{\partial_x \seps^{\alpha/2}\}_{\eps\in(0,1)}$ is bounded in $L^2((0,T)\times\R)$ so that together with the strong convergence of $\{\seps\}_{\eps\in(0,1)}$ from Theorem \ref{thm:main}, we deduce that $\partial_x \seps^{\alpha/2} \weak \partial_x s^{\alpha/2}$ weakly in $L^2((0,T)\times\R)$. In particular, $\int_0^t \int_{\R} |\partial_x s^{\alpha/2}|^2 \diff x \diff \tau \leq \liminf_{\eps\to0}\int_0^t \int_{\R} |\partial_x \seps^{\alpha/2}|^2 \diff x \diff \tau$. For the logarithmic terms, $\ueps \to u$ and $\veps \to v$ strongly in $L^1((0,T)\times\R)$ by Theorem~\ref{thm:main}, so we can extract a~subsequence such that
\begin{align*}
& \ueps(t,x) \to u(t,x), \quad 
\veps(t,x)\to v(t,x), \\ 
& u_{\eps}(t,x) \, |x|^{1/2} \to u(t,x)\,|x|^{1/2}, \quad v_{\eps}(t,x) \, |x|^{1/2} \to v(t,x)\,|x|^{1/2}
\end{align*}
strongly in $L^1(\R)$ for a.e.\ $t$. The latter two are possible since, for instance, $\ueps(t,x) \to u(t,x)$ in $L^{1}(\R)$ and $\{\ueps(t,x) \,|x|\}$ is bounded in $L^1(\R)$ for a.e.\ $t$, thus one obtains the desired convergence by splitting $\R$ into a ball $B_R$ and $\R\setminus B_R$, passing first to the limit $\eps\to0$, and then letting $R\to \infty$. By writing
\begin{equation}\label{eq:adding_exp_to_make_fcnal_lower_semicontinuous}
\int_{\R} u_{\eps}(t,x)\log u_{\eps}(t,x) \diff x = \int_{\R} \frac{u_{\eps}(t,x)}{e^{-|x|^{1/2}}}\log \frac{u_{\eps}(t,x)}{e^{-|x|^{1/2}}} (e^{-|x|^{1/2}}\diff x) - \int_{\R} u_{\eps}(t,x)\,|x|^{1/2} \diff x,
\end{equation}
we deduce that
$$
\int_{\R} u(t,x)\log u(t,x) \diff x \leq \liminf_{\eps\to0}\int_{\R} u_{\eps}(t,x)\log u_{\eps}(t,x) \diff x
$$
by the lower semicontinuity of the relative entropy with the measure $e^{-|x|^{1/2}}\diff x$ \cite[Theorem~2.34]{MR1857292} and the convergence $\ueps(t,x)\,|x|^{1/2} \to u(t,x)\,|x|^{1/2}$. As the same inequality is satisfied for $v$, we obtain \eqref{eq:entropy_individual_inequality_the_limit} from~\eqref{eq:entropy_before_sending_eps_to_zero}.

\underline{Step 2: Estimate on $\int_{\R} s(t,x) \log s(t,x) \diff x$.} Under the assumption that $u,v$ stays segregated up to time $t$, we will prove that
\begin{equation}\label{eq:step_2_proof_inequality_on_slogs}
\int_{\R} s(t,x) \log s(t,x) \diff x + \frac{4}{\alpha^2} \int_0^t \int_{\R} |\partial_x s^{\alpha/2}|^2 \diff x \diff \tau  - 2 
\int_0^t s(\tau,0) \diff \tau \geq \int_{\R} s^0(x) \log s^0(x) \diff x,
\end{equation}
where $s(\tau,0)$ is defined in Remark \ref{rem:def_s_at_x_0}. The idea is to consider the PDE for $s$,
$$
\partial_t s = \partial_x(s\,\partial_x p(s)) -\partial_x u + \partial_x v,
$$
multiply it by $\log s$, and integrate by parts. We formally find that 
\begin{equation}\label{eq:entropy_of_the_sum_with_additional_nice_term}
\begin{split}
\int_{\R} s(t,x) \log s(t,x)\diff x &+ \frac{4}{\alpha^2} \int_0^t \int_{\R} |\partial_x s^{\frac{\alpha}{2}}|^2 \diff x \diff \tau - \int_0^t \int_{\R} \frac{(u-v)\, \partial_x s}{s} \diff x \diff \tau \\
&= \int_{\R} s^0(x) \log s^0(x)\diff x.
\end{split}
\end{equation}
Notice that $s = u$ for a.e. $x\leq 0$ and $s = v$ for a.e. $x \geq 0$. We formally deduce that
\begin{equation}\label{eq:example_how_to_do_smooth_slns_trace_at_0}
\int_0^t \int_{\R} \frac{(u-v)\, \partial_x s}{s} \diff x \diff \tau = \int_0^t \int_{-\infty}^0 \partial_x s \diff x \diff \tau - \int_0^t \int_{0}^\infty \partial_x s \diff x \diff \tau = 2 
\int_0^t s(\tau,0) \diff \tau,
\end{equation}
so we arrive at \eqref{eq:step_2_proof_inequality_on_slogs}. However, the argument is not rigorous for two reasons: first, $\partial_x s$ does not exist (we only have $\partial_x s^{\alpha}$ and $\partial_x s^{\frac{\alpha}{2}}$ by Theorem \ref{thm:main}); second, $\log s$ is not an admissible test function. Moreover, one has to be careful, since the required direction of the inequality is incompatible with lower semicontinuity arguments.

We now make the argument rigorous. By Theorem \ref{thm:main}, the pair $(u, v)$ satisfies Definition \ref{def:weak_sol}. We consider the weak formulation for $s=u+v$, obtained by summing the weak formulations \eqref{eq:weak_form_H-1_u}--\eqref{eq:weak_form_H-1_v}, which reads as
\begin{equation}\label{eq:weak_formulation_for_s}
    \int_0^T \langle \partial_t s, \varphi \rangle_{H^{-1}(\R), H^1(\R)} \diff t = - \int_0^T \int_{\R} \left(\frac{1}{\alpha}\, \partial_x s^{\alpha} - u + v  \right)\, \partial_x \varphi \diff x \diff t.
\end{equation}
Note that when summing up the formulations, we used the property $u/s + v/s = 1$. This is not true when $s = 0$ because then, by our definition, $u/s = v/s = 0$. However, if $s = 0$, then $s^{\alpha} = 0$ and also $ \partial_x s^{\alpha} = 0$ a.e.\ on $\{s = 0\}$.

Let $\delta>0$ and let $L_{\delta}(s) = \max(s,\delta)$. The primitive function of $\log L_{\delta}(s)$ equals
$$
\widetilde{L}_{\delta}(s) = \begin{cases} s\, \log \delta &\mbox{ if } s< \delta,\\
s\log s - s + \delta &\mbox{ if } s \geq \delta.
\end{cases}
$$
Note that $\partial_x \log  L_{\delta}(s)  = \frac{\partial_x s}{s} \mathds{1}_{s\geq \delta}$
so $|\partial_x \log L_{\delta}(s)| \leq \frac{s^{\alpha}}{\delta^{\alpha}} \, \frac{|\partial_x s|}{s} \, \mathds{1}_{s\geq \delta} \leq \frac{1}{\alpha\delta^{\alpha}}\, |\partial_x s^{\alpha}|$, from which we deduce that $\partial_x \log  L_{\delta}(s) \in L^2((0,T)\times\R)$ for each $\delta>0$. We consider the test function $\varphi(\tau,x) = \log L_{\delta}(s)\, \psi_R(x)\, \mathds{1}_{[0,t]}(\tau)$ in \eqref{eq:weak_formulation_for_s}, where $\psi_R$ is a smooth function with $\psi_R =1$ on $[-R,R]$, $\psi_R \geq 0$, $\psi_R = 0$ on $\R \setminus [-2R,2R]$, and $|\partial_x \psi_R|\leq C_{\psi}/R$ for some constant $C_{\psi}$.

By \cite[Lemma A.1]{MR4563032}, the following chain rule holds:
\begin{equation}\label{eq:chain_rule_application_slogs_negative_H1}
\partial_t \int_{\R} \widetilde{L}_{\delta}(s(t,x)) \, \psi_R(x) \diff x = \langle \partial_t s, \log{L}_{\delta}(s(t,x)) \, \psi_R(x) \rangle_{H^{-1}(\R), H^1(\R)}.
\end{equation}
Moreover, we know $\partial_t s \in L^2(0,T; H^{-1}(\R))$  by Lemma \ref{lem_additional_prop_reg_weak_sln}, and the right-hand side of \eqref{eq:chain_rule_application_slogs_negative_H1} is in $L^1(0,T)$ for a.e. $t$ since $\log{L}_{\delta}(s(t,x)) \, \psi_R(x) \in L^2(0,T; H^1(\R))$. Therefore, the map $t \mapsto \int_{\R} \widetilde{L}_{\delta}(s(t,x)) \, \psi_R(x) \diff x$ is absolutely continuous and for a.e. $0<r<t<T$
\begin{equation}\label{eq:application_of_the_chain_rule_negative_ssob_space_for_time_derivative_LlogL}
\begin{split}
\int_{\R} \widetilde{L}_{\delta}(s(t,x)) \, &\psi_R(x) \diff x + \frac{1}{\alpha} \int_r^t \int_{\R} \partial_x s^{\alpha}\, \partial_x(\log(L_{\delta}(s))\,\psi_R) \diff x \diff \tau \\ &- \int_r^t \int_{\R} {(u-v)\, \partial_x (\log(L_{\delta}(s)) \,\psi_R)} \diff x \diff \tau = \int_{\R} \widetilde{L}_{\delta}(s(r,x)) \, \psi_R(x) \diff x.
\end{split}
\end{equation}
We now pass to the limit $r\to0$. We claim that, along some sequence $r_n\to0$, $s(r_n,x)\weaks~s^0(x)$ locally weakly$^*$ in the space of measures, i.e. $\int_{\R} \varphi(x)\, (s(r_n,x)-s^0(x)) \diff x \to 0$ for all continuous and compactly supported $\varphi$. By the Banach-Alaoglu theorem and the bound on $\int_{\R} s(r,x) \diff x$ for a.e. $r$, there exists a sequence $r_n \to 0$ such that $s(r_n,x)\weaks~\mu$ for some measure $\mu$. By Lemma \ref{lem_additional_prop_reg_weak_sln}, $s(r,x)\to s^0(x)$ in $H^{-1}(\R)$ as $r\to0$, so we deduce $\int_{\R} \mu \, \varphi \diff x = \int_{\R} s^0\, \varphi \diff x$ for smooth and compactly supported $\varphi$. Hence $\mu = s^0$ and the claim follows. Finally, since the map $[0,\infty)\ni s\mapsto \widetilde{L}_{\delta}(s)$ is convex, lower semicontinuity of convex functionals with respect to the local weak$^*$ convergence \cite[Theorem~2.34]{MR1857292} yields
\begin{equation}\label{eq:lower_semicontinuity_entropy_time_0_proof_that_trace_at_0_is_zero}
\liminf_{r\to 0} \int_{\R} \widetilde{L}_{\delta}(s(r,x)) \, \psi_R(x) \diff x \geq \int_{\R} \widetilde{L}_{\delta}(s^0(x)) \, \psi_R(x) \diff x.
\end{equation}
Here, we do not need to subtract $e^{-|x|^{1/2}}$ as in \eqref{eq:adding_exp_to_make_fcnal_lower_semicontinuous}, since the integrand is on a bounded domain thanks to $\psi_R$. 

It follows that
\begin{align*}
\int_{\R} \widetilde{L}_{\delta}(s(t,x)) \, \psi_R(x) \diff x - \int_{\R} \widetilde{L}_{\delta}(s^0(x)) \, \psi_R(x) \diff x+ \frac{1}{\alpha} \int_0^t \int_{\R} \partial_x s^{\alpha}\, \partial_x(\log L_{\delta}(s)\,\psi_R) \diff x \diff \tau \\ - \int_0^t \int_{\R} {(u-v)\, \partial_x (\log L_{\delta}(s) \,\psi_R)} \diff x \diff \tau =: A + B + C +D \geq 0. 
\end{align*}
We want to perform the limits $\delta \to 0$ and $R \to \infty$ together along a countable sequence. Clearly, by the dominated convergence (note that $|\widetilde{L}_{\delta}(s(t,x)) |\leq s(t,x)\,|\log s(t,x)| + s(t,x) \in L^1(\R)$ for a.e.\ $t$, by the regularity from Theorem~\ref{thm:main} together with a standard estimate of $s|\log s|$ in terms of $s^2$ and $s |x|$; see, e.g., \cite[Remark 2.2]{PME_1d_Brinkman_Darcy_rate}), the first and the second term converge to $\int_{\R} s(t,x) \log s(t,x) \diff x$ and $-\int_{\R} s^0 \log s^0(x) \diff x$, respectively. The term $C$ can be expanded as
\begin{align*}
C &= \frac{1}{\alpha} \int_0^t \int_{\R} \partial_x s^{\alpha}\, \partial_x(\log L_{\delta}(s) \,\psi_R) \diff x \diff \tau \\ 
&= \frac{1}{\alpha} \int_0^t \int_{\R} \partial_x s^{\alpha}\, \frac{\partial_x s}{s} \, \mathds{1}_{s>\delta}\, \psi_R \diff x \diff \tau + \frac{1}{\alpha} \int_0^t \int_{\R} \partial_x s^{\alpha}\, \log L_{\delta}(s)\,\partial_x \psi_R \diff x \diff \tau.
\end{align*}
The first term converges to $\frac{4}{\alpha^2} \int_0^t \int_{\R} |\partial_x s^{\frac{\alpha}{2}}|^2 \diff x \diff \tau$ by dominated convergence. For the second term, we estimate $|\log L_{\delta}(s)|\leq |\log\delta| + s \leq |\log \delta| + \|s\|_{L^{\infty}_{t,x}} \leq (1+\|s\|_{L^{\infty}_{t,x}})\,|\log \delta|$ for $\delta \in (0,1)$ and $|\partial_x \psi_R|\leq C_{\psi}/R$ so by the Hölder inequality, it can be estimated by
\begin{align*}
 \int_0^t \int_{\R} \partial_x s^{\alpha}\, \log L_{\delta}(s) \,\partial_x \psi_R \diff x \diff \tau &\leq \|\partial_x s^{\alpha}\|_{L^2_{t,x}} \, \frac{C_{\psi}\,(1+\|s\|_{L^{\infty}_{t,x}})\,|\log \delta|}{R}\, \|\mathds{1}_{R\leq|x|\leq2R} \|_{L^2_{t,x}}\\
 &\leq \|\partial_x s^{\alpha}\|_{L^2_{t,x}} \, \frac{C_{\psi}\,(1+\|s\|_{L^{\infty}_{t,x}})\,|\log \delta|}{R} \,  \, \sqrt{2R\, t}. 
\end{align*}
Choosing $\delta$ and $R$ in such a way that $|\log\delta|/\sqrt{R} \to 0$, the term converges to zero. Finally, for the term~$D$, 
$$
D=- \int_0^t \int_{\R} {(u-v)\, \frac{\partial_x s}{s}\, \mathds{1}_{s>\delta} \, \psi_R} \diff x \diff \tau - \int_0^t \int_{\R} {(u-v)\, \log(L_{\delta}(s)) \, \partial_x\psi_R} \diff x \diff \tau.
$$
The second term converges to 0 since $u\, |\log L_{\delta}(s)|$, $v\, |\log L_{\delta}(s)|\leq s \, |\log L_{\delta}(s)| \leq s\,|\log s|$ for $\delta \in (0,1)$ and $|\partial_x\psi_R|\leq C_{\psi}/R$. For the first term, as in \eqref{eq:example_how_to_do_smooth_slns_trace_at_0}, we use $u = s$ for $x\leq 0$ and $v = s$ for $x \geq 0$ to find that
\begin{align}\label{eq:the_first_term_in_D_to_pass_to_the_limit}
-&\int_0^t \int_{\R} (u-v)\, \frac{\partial_x s}{s}\, \mathds{1}_{s>\delta} \, \psi_R \diff x \diff \tau \\
&=- \int_0^t \int_{-\infty}^0 \partial_x \max(s,\delta)\, \psi_R \diff x  \diff \tau + \int_0^t \int_0^{\infty} \partial_x \max(s,\delta)\, \psi_R \diff x  \diff \tau \nonumber \\
&= -2\int_0^t \max(s(\tau,0),\delta) \diff \tau + \int_0^t \int_{\R} \max(s,\delta)\, \partial_x \psi_R \, (\mathds{1}_{(-\infty,0]}(x) -\mathds{1}_{[0,\infty)}(x)) \diff x \diff \tau. \nonumber 
\end{align}
Note that, for every $\delta>0$, the function $x\mapsto \max(s(\tau,x),\delta)$ admits a continuous representative for a.e.\ $\tau$, since $\partial_x\max(s,\delta) = \partial_x s \, \mathds{1}_{s\geq \delta} = \frac{1}{\alpha}\,s^{1-\alpha}\,\partial_x s^{\alpha}\, \mathds{1}_{s\geq \delta} \in L^2(\mathbb R)$ for a.e. $\tau$. The set of times $\tau$ for which this property holds simultaneously along a countable sequence $\delta\to0$ has full measure, thereby justifying the integration above. Moreover, this continuous representative coincides with $\max(s(\tau,x),\delta)$, where $s(\tau,x)$ denotes the continuous representative constructed in Remark~\ref{rem:def_s_at_x_0}. Indeed, two continuous functions that agree almost everywhere must agree everywhere.

Coming back to \eqref{eq:the_first_term_in_D_to_pass_to_the_limit}, we note that
\begin{align*}
\int_0^t \int_{\R} \max(s,\delta)\, |\partial_x \psi_R| \diff x \diff \tau &\leq \frac{C_{\psi}}{R}\int_0^t \int_{\R} s \diff x \diff \tau + \delta \int_0^t \int_{-2R}^{2R} \frac{C_{\psi}}{R} \diff x \diff \tau\\
&\leq \frac{C_{\psi}\,t}{R} \int_{\R} s^0 \diff x + 4\, C_{\psi} \, \delta\,t,  
\end{align*}
so we can pass to the limit in \eqref{eq:the_first_term_in_D_to_pass_to_the_limit} to obtain
$$
- \int_0^t \int_{\R} (u-v)\, \frac{\partial_x s}{s}\, \mathds{1}_{s>\delta} \, \psi_R \diff x \diff \tau  \to -2\int_0^t s(\tau,0) \diff \tau,
$$
which concludes the proof of \eqref{eq:step_2_proof_inequality_on_slogs}.

\underline{Step 3: Conclusion.} For a.e. $t \in [0,t^*]$, we have $\int_0^{\infty}u(t,x)\diff x = \int_{-\infty}^0 v(t,x) \diff x=0$, so that $s\log s = u\log u + v\log v$ a.e. on $\R$. We subtract \eqref{eq:step_2_proof_inequality_on_slogs} from \eqref{eq:entropy_individual_inequality_the_limit} to conclude that $\int_0^t s(\tau,0) \diff \tau =~0$. Passing to the limit $t \nearrow  t^*$, we infer that $\int_0^{t^*} s(\tau,0) \diff \tau =~0$, so $s(t,0) = 0$ for a.e. $t \in [0,t^*]$ as desired. 
\end{proof}

\begin{rem}\label{rem:difficulty_if_arguing_when_s_smooth}
In Step 2 of the proof above, one might be tempted to first establish the result for $s_\varepsilon$ and then pass to the limit $\varepsilon \to 0$, since $\seps$ enjoys additional regularity. However, this approach would require the strong convergence $\partial_x \seps^{\frac{\alpha}{2}} \to \partial_x s^{\frac{\alpha}{2}}$ in $L^2((0,T)\times\R)$, which is not known to hold in general for all $\alpha \in (0,\infty)$.
\end{rem}

We are now ready to prove Theorem \ref{thm:main_nonuniqueness}, Corollary~\ref{cor:nonuniqueness_mixing}, and Theorem \ref{thm:quantitative_bounds}.
\begin{proof}[Proof of Theorem \ref{thm:main_nonuniqueness}]

The non-mixing solution $(u_A, v_A)$ has been constructed in Proposition \ref{prop:existence_non_mixing_sln}. The mixing solution $(u_B, v_B)$ will be constructed as the vanishing viscosity limit from Theorem \ref{thm:main}. To this end, we first claim that for any regular weak solution to \eqref{eq:intro-toy-model-to-study-nonuniqueness} (see Definition \ref{def:regular_weak_sln}), $s= u+ v$ satisfies
\begin{equation}\label{eq:identity_for_the_center_of_the_mass}
\int_{-\infty}^0 s(t,x) \, x \diff x - \int_{-\infty}^0 s^0(x) \, x \diff x  = - \frac{1}{\alpha} \int_0^t s^{\alpha}(\tau,0) \diff \tau   + \int_0^t \int_{-\infty}^0 (u- v) \diff x \diff \tau.
\end{equation}
Rigorously, this identity can be obtained by using a test function of the form $\varphi(\tau,x) = \min(x,0)\, \psi_R(x) \, \mathds{1}_{[0,t]}(\tau)$ in \eqref{eq:weak_formulation_for_s} and using $\partial_x s^{\alpha} \in L^2((0,T)\times\R)$, $s\,|x| \in L^{\infty}(0,T; L^1(\Rd))$ and $s\in L^{\infty}((0,T)\times\R)$ to let $R\to\infty$, similarly as in the proof of Proposition~\ref{prop:trace_at_0_for_separated_solutions}. More precisely, the term involving $\partial_t s$ is treated using the chain rule (as in \eqref{eq:chain_rule_application_slogs_negative_H1}). In contrast to~\eqref{eq:lower_semicontinuity_entropy_time_0_proof_that_trace_at_0_is_zero}, we have $\lim_{r\to0} \int_{\R} s(r,x)\,\min(x,0)\,\psi_R \diff x= \int_{\R} s^0(x)\,\min(x,0)\,\psi_R \diff x$, so we obtain an equality in \eqref{eq:identity_for_the_center_of_the_mass} rather than an inequality as in \eqref{eq:step_2_proof_inequality_on_slogs}. The term with $u-v$ reads as $\int_0^t \int_{-\infty}^0 (u-v)\,(\psi_R + x\,\partial_x \psi_R) \diff x \diff \tau$. The first part is handled via dominated convergence, while the second one is estimated via the decay $|\partial_x \psi_{R}|\leq C_{\psi}/R$ and integrability of $u\,|x|$ and $v\,|x|$. The remaining term is more subtle. We compute
\begin{align*}
-\frac{1}{\alpha} \int_{0}^t \int_{-\infty}^0 \partial_x s^{\alpha} \,&\partial_x (\min(x,0)\, \psi_R(x)) \diff x \diff \tau = 
-\frac{1}{\alpha} \int_{0}^t \int_{-\infty}^0 \partial_x s^{\alpha} \, (x\, \partial_x \psi_R + \psi_R ) \diff x \diff \tau\\
&= -\frac{1}{\alpha} \int_{0}^t s^{\alpha}(\tau,0) \diff \tau  + \frac{1}{\alpha} \int_{0}^t \int_{-\infty}^0 s^{\alpha} \, \partial_x(x\, \partial_x \psi_R + \psi_R ) \diff x \diff \tau\\
&= -\frac{1}{\alpha} \int_{0}^t s^{\alpha}(\tau,0) \diff \tau + \frac{1}{\alpha} \int_{0}^t \int_{-\infty}^0 s^{\alpha} \, (2\,\partial_x \psi_R + x \,\partial^2_x \psi_R) \diff x \diff \tau.
\end{align*}
To prove that the last integral converges to zero when $R\to \infty$, we need a condition on $\partial^2_x \psi_R$. It is not difficult to see that there is a cutoff function such that $\psi_R=1$ on $[-R,R]$, $\psi_R$~is supported on $[-2R,2R]$, $|\partial_x\psi_R|\leq C_{\psi}/R$, $|\partial^2_x \psi_R|\leq C_{\psi}/R^2$ for a numerical constant $C_{\psi}$; see, e.g., \cite[Remark~E.2]{carrillo2026new} for an explicit construction. With this decay, the term $x\,|\partial^2_x \psi_R|$ has the same properties as $|\partial_x \psi_R|$, so we only need to study the integral $\int_{0}^t \int_{-\infty}^0 s^{\alpha} \,|\partial_x \psi_R| \diff x \diff \tau$. If~$\alpha \geq 1$, $s^{\alpha}\in L^1((0,T)\times\R)$ by Theorem~\ref{thm:main}, and thus, the integral converges to zero by $|\partial_x\psi_R|\leq C_{\psi}/R$. On the other hand, if $\alpha < 1$, we apply Hölder's inequality with exponents $1/\alpha$ and $1/(1-\alpha)$. Combining 
$$
\|\partial_x \psi_R \|_{L^{\frac{1}{1-\alpha}}_{t,x}} \leq \frac{C_{\psi}}{R} \, (t\,R)^{{1-\alpha}} \leq C_{\psi}\,t^{{1-\alpha}}\, R^{-\alpha} \to 0 \mbox{ as } R\to \infty
$$
with the $L^{\infty}(0,T;L^1(\R))$ bound on $s$, $\int_{0}^t \int_{-\infty}^0 s^{\alpha} \,|\partial_x \psi_R| \diff x \diff \tau \to 0$ and we arrive at \eqref{eq:identity_for_the_center_of_the_mass}.

Now, let $\delta>0$ and suppose that the vanishing viscosity limit yields a solution $(u,v)$ such that, for a.e. $t \leq -\int_{-\infty}^0 u^0(x) \, x \diff x/\int_{-\infty}^0 u^0(x) \diff x + \delta$, $u(t,x)$ is supported for $x\in(-\infty,0]$, while $v(t,x)$ is supported for $x\in[0,\infty)$. By the segregation assumption and Proposition~\ref{prop:trace_at_0_for_separated_solutions}, $s^{\alpha}(\tau,0) = 0$ a.e., so identity~\eqref{eq:identity_for_the_center_of_the_mass} simplifies to 
\begin{equation}\label{eq:mixing_cfinal_contradiction}
\int_{-\infty}^0 u(t,x) \, x \diff x - \int_{-\infty}^0 u^0(x) \, x \diff x  =  t \int_{-\infty}^0 u^0(x) \diff x.
\end{equation}
It follows that $\int_{-\infty}^0 u(t,x) \, x \diff x$ becomes positive for sufficiently large $t$, which leads to a contradiction and yields a~set $\mathcal{F}$ of positive measure, where the segregation assumption does not hold. In particular,
$$
\essinf \mathcal{F} \leq -\frac{\int_{-\infty}^0 u^0(x) \, x \diff x}{\int_{-\infty}^0 u^0(x) \diff x} + \delta.
$$
Sending $\delta \to 0$, we obtain estimate \eqref{eq:estimate_time_mixing_main_statement}.
\end{proof}

\begin{proof}[Proof of Corollary~\ref{cor:nonuniqueness_mixing}] Let $(u_A, v_A)$, $(u_B, v_B)$ be as in the statement of Theorem \ref{thm:main_nonuniqueness} constructed on the interval $[0,t_{m}+1]$, where $t_m>-\int_{-\infty}^0 u^0(x) \, x \diff x/\int_{-\infty}^0 u^0(x) \diff x$, so that thanks to \eqref{eq:estimate_time_mixing_main_statement}, the solution $(u_B, v_B)$ is already mixed at time $t_m$. 

Now, let us take the segregated solution $u_A(t,\cdot), v_A(t,\cdot) \in L^1(\R)\cap L^{\infty}(\R)$, $\int_{-\infty}^0 v_A(t,x) \diff x = \int_{0}^{\infty} u_A(t,x) \diff x = 0$, $u_A(t,\cdot)\,|\cdot| \in L^1(\R)$ for a.e.\ $t$, and let $t_0 \in (t_m, t_m + \frac{1}{2})$. By Theorem \ref{thm:main_nonuniqueness}, there exists a mixing solution $(u_C, v_C)$ with initial condition $(u_C(0,\cdot), v_C(0,\cdot)) = (u_A(t_0,\cdot), v_A(t_0,\cdot))$, which will mix in finite time. We define $(\tilde{u}, \tilde{v})$ as 
$$
\tilde{u}(t,x) = \begin{cases}
u_A(t,x) &\mbox{ if } t <t_0,\\
u_C(t-t_0,x) &\mbox{ if } t\geq t_0,
\end{cases} \qquad 
\tilde{v}(t,x) = \begin{cases}
v_A(t,x) &\mbox{ if } t <t_0,\\
v_C(t-t_0,x) &\mbox{ if } t\geq t_0.
\end{cases}
$$
We also define $s_A$, $s_C$, and $\tilde{s}$ in an obvious way. By construction, $(\tilde{u}, \tilde{v})$ gets mixed for some time greater than $t_0>t_m$, so it is distinct from $(u_B,v_B)$. Hence, it only remains to prove that the solution is a regular weak solution for $t \geq t_0$ as stated in Definition~\ref{def:regular_weak_sln}. Clearly, all the regularity assertions therein are satisfied, so we only need to prove \eqref{eq:weak_form_original_form_u}--\eqref{eq:weak_form_original_form_v}. By \eqref{eq:weak_formulation_H^-1,H^1_pointwise_in_time}, for all $\psi \in C_c^{\infty}(\Rd)$ and for all $t \in [0,t_0]$,
$$
- \langle u_A(t, \cdot),  \psi \rangle_{H^{-1}(\R), H^1(\R)} + \langle u^0,  \psi \rangle_{H^{-1}(\R), H^1(\R)}= \int_0^t \int_{\R} \left(\frac{1}{\alpha} \, \frac{u_A}{s_A}\, \partial_x s^{\alpha}_A - u_A \right)\, \partial_x \psi \diff x \diff \tau,
$$
while for all $t \geq t_0$:
\begin{align*}
- &\langle u_C(t-t_0, \cdot),  \psi \rangle_{H^{-1}(\R), H^1(\R)} + \langle u_A(t_0,\cdot),  \psi \rangle_{H^{-1}(\R), H^1(\R)}\\ 
&= \int_{0}^{t-t_0} \int_{\R} \left(\frac{1}{\alpha} \, \frac{u_C}{s_C}\, \partial_x s^{\alpha}_C - u_C \right)\, \partial_x \psi \diff x \diff \tau = \int_{t_0}^{t} \int_{\R} \left(\frac{1}{\alpha} \, \frac{\tilde{u}}{\tilde{s}}\, \partial_x \tilde{s}^{\alpha} - \tilde{u} \right)\, \partial_x \psi \diff x \diff \tau.
\end{align*}
Combining both equations we obtain for all $t\geq 0$ 
$$
- \langle \tilde{u}(t, \cdot),  \psi \rangle_{H^{-1}(\R), H^1(\R)} + \langle u^0,  \psi \rangle_{H^{-1}(\R), H^1(\R)}= \int_0^t \int_{\R} \left(\frac{1}{\alpha} \, \frac{\tilde{u}}{\tilde{s}}\, \partial_x \tilde{s}^{\alpha} - \tilde{u} \right)\, \partial_x \psi \diff x \diff \tau.
$$
Multiplying this equation by $\partial_t \eta$, where $\eta \in C^{\infty}[0,T]$ with $\eta(T) = 0$, and integrating in time, we obtain
$$
\int_0^T \int_{\R} \tilde{u}\, \partial_t \eta \, \psi \diff x \diff \tau + \int_{\R} u^0(x) \, \eta(0)\, \psi(x) \diff x = \int_0^T \int_{\R} \left(\frac{1}{\alpha} \, \frac{\tilde{u}}{\tilde{s}}\, \partial_x \tilde{s}^{\alpha} - \tilde{u} \right)\, \partial_x \psi \, \eta \diff x \diff \tau.
$$
Since functions of the form $\psi(x) \, \eta(t)$ with $\eta(T) = 0$ can approximate any function in $C_c^{\infty}([0,T) \times\R)$, we infer that $\tilde{u}$ satisfies \eqref{eq:weak_form_original_form_u}. The same argument works for $\tilde{v}$. It follows that $(\tilde{u},\tilde{v})$ is a regular weak solution. It is also clear that the argument can be used to construct another mixing solution by considering a solution that stays segregated up to a moment when $(\tilde{u}, \tilde{v})$ already got mixed and gluing it with a mixing solution. It follows that there are infinitely many mixing solutions.

We now discuss the uniqueness of the segregated solution. Let $(u,v)$ be a segregated solution on $\R^-=(-\infty,0]$ and $\R^+=[0,\infty)$. Then $s=u$ a.e.\ on $(-\infty,0]$. Hence, $u$ is a weak solution to~\eqref{eq:Neumann_problem_on_a_half_space} with Neumann boundary conditions and initial condition $u_0$. Moreover, it satisfies $u\in L^{\infty}(0,T; L^1(\R^-)) \cap L^{\infty}(0,T; L^\infty(\R^-))$, $u|x|\in L^{\infty}(0,T; L^1(\R^-))$, $\partial_x u^\alpha \in L^2((0,T)\times\R^-)$, $\partial_t u \in L^2(0,T;H^{-1}(\R^-))$. Such a solution is unique; see, for example, \cite[Proposition~3.5]{MR2286292}.
\end{proof}

\begin{proof}[Proof of Theorem~\ref{thm:quantitative_bounds}] We remind that throughout the proof we assume additionally that $v(t,x) = u(t,-x)$ a.e.\ and $\alpha \in (1,2)$. In what follows, the constants $C_1(\alpha, s^0)$, $C_2(\alpha, s^0)$, $C_3(\alpha, s^0)$, $C_4(\alpha, s^0)$, $C(\alpha, s^0)$ will depend on $\alpha$, $\|s^0\|_{L^1(\R)}$, $\|s^0\|_{L^{\infty}(\R)}$, and $\int_{\R} s^0\, |x|^2 \diff x$. 

\underline{Step 1: The entropy estimate.} We will prove that there exists a constant $C_1(\alpha, s^0)$ such that
$$
-C_1(\alpha, s^0)\, (1+\log (t+1)) + \int_0^t s(\tau,0) \diff \tau \leq 2 \int_0^t \int_{\R^-} {v^{2-\alpha}} \diff x \diff \tau,
$$
where $\R^-=(-\infty,0]$. We consider functional $F(t) = \int_{\R} \left(u\log u + v\log v - \frac{1}{2} s\log s\right) \diff x$. Similarly to the derivation of the rigorous versions of \eqref{eq:entropy_individual_inequality_the_limit} and \eqref{eq:entropy_of_the_sum_with_additional_nice_term} in Proposition~\ref{prop:trace_at_0_for_separated_solutions}, we obtain the inequality
$$
F(t) + \frac{2}{\alpha^2} \int_0^t \int_{\R}  |\partial_x s^{\alpha/2}|^2 \diff x \diff \tau \leq - \frac{1}{\alpha}\int_0^t\int_{\R} \frac{u-v}{s^{\alpha/2}}\, \partial_x s^{\alpha/2} \diff x \diff \tau + F(0).
$$
The only difference in the derivation is that, instead of \eqref{eq:the_first_term_in_D_to_pass_to_the_limit}, we write
$$
-\frac{1}{2} \int_0^t \int_{\R} (u-v)\, \frac{\partial_x s}{s}\, \mathds{1}_{s>\delta} \, \psi_R \diff x \diff \tau = - \frac{1}{\alpha}\int_0^t\int_{\R} \frac{u-v}{s^{\alpha/2}}\, \partial_x s^{\alpha/2}\, \mathds{1}_{s>\delta}\, \psi_R \diff x \diff \tau
$$
(comparing with \eqref{eq:the_first_term_in_D_to_pass_to_the_limit}, the factor of $\frac{1}{2}$ appears in front of the expression on the left-hand side because we take $\frac{1}{2} s\log s$ in the definition of $F$). The joint limit $R\to \infty$, $\delta\to0$ is easy to justify by dominated convergence, since $\partial_x s^{\alpha/2} \in L^2((0,T)\times\R)$ by Theorem~\ref{thm:main}, while $(u-v)s^{-\alpha/2} \in L^{\infty}(0,T;L^2(\R))$ for all $\alpha <2$, because $|u-v|s^{-\alpha/2} \leq s^{1-\alpha/2}$ and $s$ has arbitrarily high moments bounded by Lemma \ref{lem:arbitrary_high_moments}.

Since $u(t,x)=v(t,-x)$, we find that
\begin{align*}
- \frac{1}{\alpha}\int_0^t\int_{\R} \frac{u-v}{s^{\alpha/2}}\, \partial_x &s^{\alpha/2} \diff x \diff \tau = - \frac{2}{\alpha}\int_0^t\int_{\R^-} \frac{u-v}{s^{\alpha/2}}\, \partial_x s^{\alpha/2} \diff x \diff \tau \\ &= - \int_0^t s(\tau,0) \diff \tau +  \frac{4}{\alpha}\int_0^t\int_{\R^-} \frac{v}{s^{\alpha/2}}\, \partial_x s^{\alpha/2} \diff x \diff \tau \\
&\leq - \int_0^t s(\tau,0) \diff \tau + 2 \int_0^t \int_{\R^-} \frac{v^2}{s^{\alpha}} \diff x \diff \tau+ \frac{2}{\alpha^2} \int_0^t \int_{\R^-} |\partial_x s^{\alpha/2}|^2 \diff x \diff \tau.
\end{align*}
It follows that
$$
F(t) - F(0) + \int_0^t s(\tau,0) \diff \tau \leq 2 \int_0^t \int_{\R^-} \frac{v^2}{s^{\alpha}} \diff x \diff \tau \leq 2\int_0^t\int_{\R^-} {v^{2-\alpha}} \diff x \diff \tau.
$$
Combining Lemma \ref{eq:ulogu_lower_bound_in_terms_of_2nd_moment} and the second estimate in Lemma \ref{lem:quantitative_dependence_constants}, we deduce that there is a~constant $C_1(\alpha, s^0)$ such that $F(t) - F(0) \geq -C_1(\alpha, s^0)\, (1+\log (t+1))$. The conclusion of this step follows. 

\underline{Step 2: The tail estimate.} We will prove that
$$
0 \leq 2 \int_{\R^-} v(t,x)\, |x| \diff x + 2\alpha^2 \int_0^t \int_{\R^-} v(\tau,x) \diff \tau \diff x  + 2 \int_0^t s^{\alpha}(\tau,0) \diff \tau - t \int_{\R} s^0(x) \diff x + E(0),
$$
where $E(t) := \frac{1}{\alpha}\int_{\R} s^{\alpha} \diff x - \int_{\R} (u-v)\,x \diff x$. Indeed, computing formally, 
\begin{equation}\label{eq:rate_of_change_energy_lower_bound_on_mixing}
\begin{split}
E'(t) &= \int_{\R} s^{\alpha-1} \partial_x(s\,\partial_x p(s) - u + v) \diff x - \int_{\R} x\, \partial_x((u-v)\,\partial_x p(s) - s) \diff x \\
&=- (\alpha-1) \int_{\R} s\,|\partial_xp(s)|^2 \diff x + \alpha \int_{\R} (u-v)\, \partial_x p(s) \diff x- \int_{\R} s^0(x) \diff x.
\end{split}
\end{equation}
Since we are only interested in an inequality, the above computation can be made rigorous by performing it for the approximate solutions to \eqref{eq:general_cross_diffusion_intro_viscosity} and then passing to the limit $\eps\to0$ using Theorem~\ref{thm:main}. By the assumption, $(u-v)\,\partial_x p(s)$ and $s\,|\partial_x p(s)|^2$ are even functions so that 
\begin{align*}
\int_{\R} (u-v)\, \partial_x p(s) \diff x &= 2\int_{\R^-} (u-v) \, \partial_x p(s) \diff x = 2\int_{\R^-} s \, \partial_x p(s) \diff x - 4\int_{\R^-} v\, \partial_x p(s) \diff x\\
&=\frac{2}{\alpha} s^{\alpha}(t,0) + 4\int_{\R^-} \frac{v}{\sqrt{s}}\, \sqrt{s}\,|\partial_x p(s)| \diff x \\
&\leq \frac{2}{\alpha} s^{\alpha}(t,0) + 2\, \alpha \int_{\R^-} v\diff x + \frac{\alpha-1}{\alpha} \int_{\R} s\,|\partial_x p(s)|^2 \diff x.
\end{align*}
Combining it with \eqref{eq:rate_of_change_energy_lower_bound_on_mixing}, we find that 
\begin{equation}\label{eq:ODI_for_energy_bound_on_the_mixing_quantity}
E'(t) \leq 2\, s^{\alpha}(t,0) + 2\, \alpha^2 \int_{\R^-} v\diff x - \int_{\R} s^0(x) \diff x.
\end{equation}
On the other hand, using that $(u-v)\,x$ is an even function,
$$
E(t) \geq -2 \int_{\R^-} (u-v)\, x \diff x \geq - 2\int_{\R^-} v\, |x| \diff x.
$$
Integrating \eqref{eq:ODI_for_energy_bound_on_the_mixing_quantity} in time and using the lower bound above, we obtain the desired estimate in this step.

\underline{Step 3: Conclusion.} Suppose that $t$ is such that $
2 \int_0^t s^{\alpha}(\tau,0) \diff \tau + E(0) + 1\leq \frac{t}{2} \int_{\R} s^0(x) \diff x$. Then, by Step 2,
$$
\frac{t}{2} \int_{\R} s^0(x) \diff x + 1 \leq 2 \int_{\R^-} v(t,x)\, |x| \diff x + 2\alpha^2 \int_0^t \int_{\R^-} v(\tau,x) \diff \tau \diff x,
$$
which can be rewritten as
$$
C_2(\alpha,s^0)\, (1+t) \leq  2 \int_{\R^-} v(t,x)\, |x| \diff x + 2\alpha^2 \int_0^t \int_{\R^-} v(\tau,x) \diff \tau \diff x
$$
for some constant $C_2( \alpha,s^0)$, which proves \eqref{eq:quant_bound_main_thm_energy_one}. On the other hand, if $t$ is such that $
2 \int_0^t s^{\alpha}(\tau,0) \diff \tau + E(0) + 1\geq \frac{t}{2} \int_{\R} s^0(x) \diff x$, we have
\begin{align*}
\frac{t}{2} \int_{\R} s^0(x) \diff x &\leq 2 \int_0^t s^{\alpha}(\tau,0) \diff \tau  + E(0) + 1 \leq 2\, \| s\|_{L^{\infty}((0,t)\times\R)}^{\alpha-1} \int_0^t s(\tau,0) \diff \tau + E(0) + 1 \\ &\leq C(s^0,\alpha)\, (1+t^{2(\alpha-1)/(3\alpha)})\, \int_0^t  s(\tau,0) \diff \tau + E(0) + 1,
\end{align*}
where constant $C(\alpha,s^0)$ comes from the first estimate in Lemma~\ref{lem:quantitative_dependence_constants}. We deduce from $1 - 2(\alpha-1)/(3\alpha) = (\alpha+1)/(3\alpha)$ that
$$
\int_0^t s(\tau, 0) \diff \tau \geq C_3(\alpha,s^0) \, t^{(\alpha+1)/(3\alpha)} - C_4(\alpha,s^0). 
$$
By Step 1, we obtain \eqref{eq:quant_bound_main_thm_log_one} and the proof concludes.
\end{proof}

\begin{rem}
More quantitative information can be deduced if the solution satisfies additional conditions. For example, if $s(\tau,0) = 0$ for a.e. $0<\tau<t$, then we deduce from Step~2 of the proof of Theorem~\ref{thm:quantitative_bounds} that 
$$
  t\int_{\R} s^0(x) \diff x - E(0) \le 2 \int_{\R^-} v(t,x)\, |x| \diff x + 2\,\alpha^2 \int_0^t \int_{\R^-} v(\tau,x) \diff \tau \diff x,
$$
which measures how much mass of $v$ appears on $\R^-$. This gives another estimate on the first time of mixing under the assumptions of Theorem~\ref{thm:quantitative_bounds} as
    $$
    \essinf \mathcal{F} \leq \frac{\frac{1}{\alpha}\int_{\R} (s^0)^{\alpha} \diff x + \int_{\R} (v^0-u^0)\,x \diff x}{\int_{\R} s^0(x) \diff x}.
    $$
\end{rem}

\subsection*{Acknowledgements}
JAC was supported by the Advanced Grant Nonlocal-CPD (Nonlocal PDEs for Complex Particle Dynamics: Phase Transitions, Patterns and Synchronization) of the European Research Council Executive Agency (ERC) under the European Union’s Horizon 2020 research and innovation programme (grant agreement No. 883363). AJ acknowledges partial support from the Austrian Science Fund (FWF), grant number 10.55776/PAT2687825, and from the Austrian Federal Ministry for Women, Science and Research and implemented by \"OAD, project MULT09/2025. This work has received funding from the European Research Council (ERC) under the European Union's Horizon 2020 research and innovation programme, ERC Advanced Grant NEUROMORPH, no.~101018153.
JS was partially supported by the National Science Centre, Poland (agreement no. 2021/43/B/ST1/02851). YY was partially supported by the MOE Tier 1 grant and the Asian Young Scientist Fellowship.

\appendix

\section{Proof of the estimates \eqref{eq:regularity_half_line_solution_u}--\eqref{eq:regularity_half_line_solution_u_2}}\label{app:estimates}
We consider solution to \eqref{eq:Neumann_problem_on_a_half_space} on $[0,T] \times (-\infty,0] = [0,T]\times\R^-$ with Neumann boundary condition $u \, \partial_x p(u) - u = \frac{1}{\alpha}\,\partial_x u^{\alpha} - u= 0$ at $x = 0$ and the initial condition $u^0$. By a~standard approximation, we may assume that the solution is smooth and conserves the mass so that $M = \int_{\R^-} u(t,x) \diff x = \int_{\R^-} u^0(x) \diff x$ (the conservation will be also true after passing to the limit with the approximation by the estimate on $u\,|x|$).

\underline{Estimates on $u\,|x|$ and $\partial_x u^{\frac{\alpha}{2}}$.} We first observe that
\begin{align}\label{eq:appendix_first_moment_diffusion_half_line}
\partial_t \int_{\R^-} u\,|x| \diff x &=  \int_{\R^-} \Big(\frac{1}{\alpha} \partial_x u^{\alpha}-u\Big) \diff x =  \frac{1}{\alpha} u^{\alpha}(t,0) - \int_{\R^-} u(t,x) \diff x, \\
\label{eq:appendix_entropy_diffusion_half_line}
\partial_t \int_{\R^-} u\,\log u \diff x &= -\int_{\R^-} \Big(\frac{1}{\alpha} \partial_x u^{\alpha}-u\Big) \, \frac{\partial_x u}{u} \diff x = -\frac{4}{\alpha^2} \int_{\R^-}|\partial_x u^{\frac{\alpha}{2}}|^2 \diff x + u(t,0). 
\end{align}
We compute 
\begin{align*}
\frac{2}{\alpha+1} u^{(\alpha+1)/2}(t,0) 
&= \int_{\R^-} u^{(\alpha-1)/2}\, \partial_x u \diff x = \frac{2}{\alpha} \int_{\R^-} u^{1/2} \, \partial_x u^{\alpha/2} \diff x \\
&\leq 
\frac{2}{\alpha}\, M^{\frac{1}{2}} \, \left( \int_{\R^-} |\partial_x u^{\alpha/2}|^2 \diff x\right)^{1/2}.
\end{align*}
It follows that for some $C=C(\alpha,M)$ we have $u^{\alpha+1}(t,0) \leq C\, \int_{\R^-} |\partial_x u^{\alpha/2}|^2 \diff x$. In particular, for all $\delta \in (0,1)$,
\begin{align*}
u(t,0)&\leq C^{1/(\alpha+1)}\, \left(\int_{\R^-} |\partial_x u^{\alpha/2}|^2 \diff x\right)^{1/(\alpha+1)} \leq C^{1/\alpha}\, \delta^{-(\alpha+1)/\alpha} +  \delta^{1+\alpha} \int_{\R^-} |\partial_x u^{\alpha/2}|^2 \diff x, \\
u^{\alpha}(t,0)&\leq C^{\frac{\alpha}{\alpha+1}}\, \left(\int_{\R^-} |\partial_x u^{\alpha/2}|^2 \diff x\right)^{\alpha/(\alpha+1)} \leq C^{\alpha}\, \delta^{-(\alpha+1)} +  \delta^{(\alpha+1)/\alpha} \int_{\R^-} |\partial_x u^{\alpha/2}|^2 \diff x.  
\end{align*}
Hence, choosing $\delta>0$ sufficiently small, we combine \eqref{eq:appendix_first_moment_diffusion_half_line}--\eqref{eq:appendix_entropy_diffusion_half_line} to obtain
$$
\partial_t \int_{\R^-} u\,|x| \diff x + \partial_t \int_{\R^-} u\,\log u \diff x + \frac{2}{\alpha^2} \int_{\R^-}|\partial_x u^{\alpha/2}|^2 \diff x \leq C
$$
for some $C=C(\alpha,M)$. We integrate this inequality in time and use inequality \eqref{eq:algebraic_inequalities_on_rho_log_rho} to conclude the bounds on $u\,|x|$ and $\partial_x u^{\alpha/2}$.

\underline{Estimates on $u$ in $L^{\infty}(0,T; L^1(\R^-))$, $L^{\infty}(0,T; L^{\infty}(\R^-))$.} The $L^{\infty}(0,T; L^1(\R^-))$ bound is trivial, so we focus on the bound in $L^{\infty}(0,T; L^{\infty}(\R^-))$. For $k > 1$, we have
\begin{align*}
\partial_t \int_{\R^-} u^{k} \diff x &= k\int_{\R^-} u^{k-1}\, \partial_t u \diff x = - k\,(k-1) \int_{\R^-} u^{k-2}\,\partial_x u\, \Big(\frac{1}{\alpha} \partial_x u^{\alpha} - u\Big) \diff x\\
&\leq - k\,(k-1) \int_{\R^-} u^{k+\alpha-3}\,|\partial_x u|^2 \diff x + k\,(k-1)\, \int_{\R^-} u^{k-1}\, \partial_x u \diff x\\
&\leq - \frac{k\,(k-1)}{2} \int_{\R^-} u^{k+\alpha-3}\,|\partial_x u|^2 \diff x + \frac{k\,(k-1)}{2}\, \int_{\R^-} u^{k-\alpha+1} \diff x.
\end{align*}
At this point, we are in exactly the same situation as in Step 5 of Proposition~3.1 in \cite{skrzeczkowski2026global} (see inequality (3.12) therein). The argument then proceeds identically: using the previous inequality together with the estimate on $\partial_x u^{\frac{\alpha}{2}}$ in $L^2((0,T)\times\R^-)$, the Alikakos' iterations yield the desired $L^{\infty}((0,T)\times\R^-)$ estimates.

\underline{Estimate on $\partial_x u^{\alpha}$.} This follows now immediately from the $L^{\infty}((0,T)\times\R^-)$ bound on $u$ and the $L^2((0,T)\times\R^-)$ bound on $\partial_x u^{\frac{\alpha}{2}}$. 

\section{A priori estimates with explicit constants}

\begin{lem}\label{lem:quantitative_dependence_constants} Let $\alpha>1$, $\ueps, \veps$ be the solution to \eqref{eq:general_cross_diffusion_intro_viscosity}. There exists a constant $C=C(\alpha,s^0)$ depending on initial condition (its $L^1(\R)$, $L^{\infty}(\R)$ norm as well as $\int_{\R} s^0\,|x|^2\diff x$) and $\alpha$ such that
        $$
        \|\seps\|_{L^{\infty}((0,T)\times\R)} \leq C(\alpha,s^0)\,(1+T^{2/(3\alpha)}), \qquad 
        \int_{\R} \seps(T,x) \, |x|^2 \diff x \leq C(\alpha, s^0)\, (1+T^2).
        $$
The same estimates are valid for the limit $s$ in Theorem~\ref{thm:main}.
\end{lem}
The assumption $\alpha>1$ is made for the sake of simplicity. While it is also possible to track the time dependence when $\alpha\leq1$, the estimates are of greatest interest in the regime $\alpha>1$ since we use them to establish quantitative estimates in Theorem~\ref{thm:quantitative_bounds}. We therefore restrict our attention to this case, where certain cancellations lead to a simpler relationship.
\begin{proof}[Proof of Lemma~\ref{lem:quantitative_dependence_constants}]
We multiply the PDE for $\seps$ by $\gamma\,\seps^{\gamma-1}$
\begin{align*}
\partial_t \int_{\R} \seps^{\gamma} \diff x +\gamma\, (\gamma-1) \int_{\R} \seps^{\alpha+\gamma-3}\,|\partial_x \seps|^{2} \diff x &\leq  \gamma \int_{\R} (\ueps+ \veps) \, |\partial_x \seps^{\gamma-1}| \diff x\\
&\leq \gamma\,(\gamma-1)\, \int_{\R} \seps^{\gamma-1} \, |\partial_x \seps| \diff x,
\end{align*}
where we neglected the term $\eps \int_{\R} \partial_x \seps\, \partial_x \seps^{\gamma-1} \diff x \geq 0$. We write $\seps^{\gamma-1} = \seps^{(\alpha+\gamma-3)/2}\, \seps^{(-\alpha+\gamma+1)/2}$, so applying Cauchy-Schwarz
$$
\int_{\R} \seps^{\gamma-1} \, |\partial_x \seps| \diff x \leq  \frac{1}{2} \int_{\R} \seps^{\alpha+\gamma-3}\,|\partial_x \seps|^{2} \diff x + \frac{1}{2}\int_{\R} \seps^{-\alpha+\gamma+1} \diff x.
$$
Combining the two estimates above we obtain
\begin{equation}\label{eq:evolution_of_the_gamma_moment_before_doing_the_interpolation_without_the_rubbish}
\partial_t \int_{\R} \seps^{\gamma} \diff x  + \frac{\gamma\, (\gamma-1)}{2} \int_{\R} \seps^{\alpha+\gamma-3}\,|\partial_x \seps|^{2} \diff x \leq \frac{\gamma\,(\gamma-1)}{2} \int_{\R} \seps^{-\alpha+\gamma+1} \diff x.
\end{equation}
We estimate
\begin{equation}\label{eq:estimate_Linf_bound_integral_-alph+gamma_first_splitting}
 \int_{\R} \seps^{-\alpha+\gamma+1} \diff x = \int_{\R} \seps^{(\gamma+\alpha)/2}\, \seps^{(\gamma + 2-3\alpha)/2} \diff x \leq \| \seps^{(\gamma+\alpha)/2} \|_{L^{\infty}(\R)} \, \int_{\R} \seps^{(\gamma-3\alpha + 2)/2} \diff x.
\end{equation}
The term $\| \seps^{(\gamma+\alpha)/2} \|_{L^{\infty}(\R)}$ can be controlled by means of the dissipation. Indeed, 
\begin{equation}\label{eq:estimate_L_inf_norm_term_half_of_gamma_via_dissipation}
\| \seps^{(\gamma+\alpha)/2} \|_{L^{\infty}(\R)} \leq  \| \partial_x \seps^{(\gamma+\alpha)/2} \|_{L^{1}(\R)} \leq \frac{\gamma+\alpha}{2}\, \|\seps^{1/2} \|_{L^2(\R)} \, \| \seps^{(\gamma+\alpha-3)/2}\, \partial_x \seps \|_{L^2(\R)}.
\end{equation}
Therefore, with $M_s = \int_{\R} s^0 \diff x$, we obtain
\begin{align*}
 \int_{\R} \seps^{-\alpha+\gamma+1} \diff x & \leq \frac{\gamma+\alpha}{2} \, M_s^{1/2} \, \| \seps^{(\gamma+\alpha-3)/2}\, \partial_x \seps \|_{L^2(\R)} \, \int_{\R} \seps^{(\gamma-3\alpha + 2)/2} \diff x \\
&\leq \frac{1}{2}\,  \| \seps^{(\gamma+\alpha-3)/2}\, \partial_x \seps \|_{L^2(\R)}^2 + \frac{M_s(\gamma+\alpha)^2}{8}  \left( \int_{\R} \seps^{(\gamma-3\alpha + 2)/2} \diff x \right)^2.
\end{align*}
Coming back to \eqref{eq:evolution_of_the_gamma_moment_before_doing_the_interpolation_without_the_rubbish}, we obtain
\begin{align}\label{eq:evolution_gamma_moment_last_step_before_Holder}
\partial_t \int_{\R} \seps^{\gamma} \diff x  
&\leq \frac{M_s\gamma\,(\gamma-1)(\gamma+\alpha)^2}{16}  \left( \int_{\R} \seps^{(\gamma-3\alpha + 2)/2} \diff x \right)^2 \\
&\leq \frac{7^2 M_s\gamma^4}{6^2 \cdot 16}  \left( \int_{\R} \seps^{(\gamma-3\alpha + 2)/2} \diff x \right)^2 \leq \frac{M_s\gamma^4}{10}  \left( \int_{\R} \seps^{(\gamma-3\alpha + 2)/2} \diff x \right)^2, \nonumber
\end{align}
where the second estimate follows by assuming that $6\,\alpha \leq \gamma$ so that $\gamma + \alpha \leq 7\gamma/6$. We finally use Hölder's inequality with respect to the measure $\seps \diff x$ to obtain
\begin{equation}\label{eq:Holder_to_adjust_the_exponent_to_gamma_over_2_for_alikakos}
\int_{\R} \seps^{(\gamma-3\alpha+2)/2} \diff x \leq \left(\int_{\R} \seps^{\gamma/2} \diff x\right)^{\beta} \, M_s^{1-\beta}, \qquad \qquad \beta=\frac{(\gamma-3\alpha+2)/2-1}{\gamma/2-1},
\end{equation}
where $\beta \in [\frac{1}{2},1]$ since $\gamma\geq 6\alpha\geq  2(3\alpha-1)$. From \eqref{eq:evolution_gamma_moment_last_step_before_Holder} we deduce
\begin{align*}
\partial_t \int_{\R} \seps^{\gamma} \diff x &\leq \frac{M_s\, \gamma^4}{10} \, \left( \int_{\R} \seps^{\gamma/2} \diff x \right)^{2\beta}\, M_s^{2(1-\beta)}\leq  \frac{M_s^{3-2\beta}}{10} \, \left(1 +  \gamma^{4/\beta} \,\left( \int_{\R} \seps^{\gamma/2} \diff x \right)^{2}\right).
\end{align*}
We finally observe that $M_s^{3-2\beta} \leq 1+ M_s^3$ and $\gamma \geq 1$, $\beta \geq \frac{1}{2}$ imply that $\gamma^{\frac{4}{\beta}} \leq \gamma^8$ so that
\begin{equation}\label{eq:final_bound_to_start_alikakos}
\partial_t \int_{\R} \seps^{\gamma} \diff x \leq \frac{1+M_s^{3}}{10} \, \left(1 +  \gamma^{8} \,\left( \int_{\R} \seps^{\gamma/2} \diff x \right)^{2}\right).
\end{equation}
We define $M({\gamma})= \sup_{t\in [0,T]} \left(\int_{\R} \seps^{\gamma} \diff x\right)^{\frac{1}{\gamma}}$. Integrating in time and letting $C_{\text{num}} := \frac{1+M_s^{3}}{10}$, we obtain
$$
M(\gamma)^{\gamma} \leq \int_{\R} (s^0)^{\gamma} \diff x + C_{\text{num}} \, T \,\left(1+\gamma^{8} \, M\left( \frac{\gamma}{2}\right)^{\gamma}\right).
$$
In particular,
$$
\max\left(1,M({\gamma})^{\gamma}\right) \leq 1+\int_{\R} (s^0)^{\gamma} \diff x + 2\,C_{\text{num}} \, T\,\gamma^{8}\,\max\left(1,M\left( \frac{\gamma}{2}\right)^{\gamma}\right).
$$
We want to iterate this inequality. We let $A:= \sup_{p\in[1,\infty]} \left(1+\int_{\R} (s^0)^{p} \diff x \right)^{1/p}$, $B:=2\,C_{\text{num}} \, T$ so that
$$
\max\left(1,M({\gamma})^{\gamma}\right) \leq A^{\gamma} + B\,\gamma^{8}\,\max\left(1,M\left( \frac{\gamma}{2}\right)^{\gamma}\right).
$$
In particular, as $A\geq 1$ and $\gamma\geq 1$, 
$$
\max\left(A^{\gamma},M( {\gamma})^{\gamma}\right) \leq \gamma^{8}\,(1+ B)\,\max\left(A^{\gamma},M\left( \frac{\gamma}{2}\right)^{\gamma}\right),
$$
so after raising to the power $1/\gamma$,
$$
\max\left(A,M( {\gamma})\right) \leq \gamma^{8/\gamma}\,(1+ B)^{1/\gamma}\,\max\left(A,M\left( \frac{\gamma}{2}\right)\right).
$$
Letting $\gamma_0 = 3\alpha$ and $W_n = \max\left(A,M( 2^n\,{\gamma_0})\right)$ we obtain
\begin{equation}\label{eq:Alikakos_final_estimate_in_terms_of_W}
W_n \leq W_0 \prod_{j=1}^n (2^j\,\gamma_0)^{8/(2^j \gamma_0)}\, (1+B)^{1/(2^j \gamma_0)}.
\end{equation}
Note that the restriction $\gamma \geq 6\alpha$ that we made when deriving \eqref{eq:final_bound_to_start_alikakos} is satisfied by the choice $\gamma_0 = 3\alpha$. We now compute the constant 
$$
\log \left(\prod_{j=1}^n (2^j \,\gamma_0)^{8/(2^j \gamma_0)}\, (1+B)^{1/(2^j \gamma_0)}\right) = \sum_{j=1}^n \frac{8\, \log(2^j\,\gamma_0)}{2^j\,\gamma_0} + \sum_{j=1}^n \frac{\log(1+B)}{2^j\,\gamma_0} . 
$$
The first sum is bounded by a numerical constant while the second by $\log((1+B)^{1/\gamma_0})$. Plugging back to \eqref{eq:Alikakos_final_estimate_in_terms_of_W} and letting $n\to\infty$, we deduce
$$
\|\seps\|_{L^{\infty}((0,T)\times\R)} \leq C(s^0,\alpha)\,(1+ M(\gamma_0))\,(1+T^{1/(3\alpha)}).
$$
From \eqref{eq:evolution_gamma_moment_last_step_before_Holder} with $\gamma_0 = 3\alpha$, we deduce $M(\gamma_0) \leq C(s^0, \alpha)\,(1+T^{1/(3\alpha)})$, which concludes the proof of the first estimate. For the second one, we first use \eqref{eq:evolution_of_the_gamma_moment_before_doing_the_interpolation_without_the_rubbish} with $\gamma = \alpha$ to obtain
$$
\partial_t \int_{\R} \seps^{\alpha} \diff x  + \frac{\alpha\, (\alpha-1)}{2} \int_{\R} \seps^{2\alpha-3}\,|\partial_x \seps|^{2} \diff x \leq \frac{\alpha\,(\alpha-1)}{2} \int_{\R} \seps \diff x = \frac{\alpha\,(\alpha-1)}{2}\, M_s.
$$
We recognize that $\int_{\R} \seps^{2\alpha-3}\,|\partial_x \seps|^{2} \diff x  = \int_{\R} \seps\,|\partial_xp(\seps)|^2 \diff x$ so after integrating in time, we deduce
\begin{equation}\label{eq:estimate_used_to_prove_the_bound_on_the_2nd_moment}
\int_0^T \int_{\R} \seps\,|\partial_xp(\seps)|^2 \diff x \leq C(\alpha,s^0)\, T.
\end{equation}
Now, we multiply the PDE for $\seps$ by $|x|^2$ to obtain
\begin{align*}
\partial_t &\int_{\R} \seps\,|x|^2 \diff x \leq - 2\int_{\R} \sqrt{\seps} \, \partial_x p(\seps) \, \sqrt{\seps} \, x \diff x + \int_{\R} \seps\,|x| \diff x + \eps \, M_s\\
&\leq 2 \, \left(\int_{\R} \seps \, |\partial_x p(\seps)|^2 \diff x \right)^{1/2} \left(\int_{\R} \seps \, |x|^2 \diff x\right)^{1/2} + M_s^{\frac{1}{2}}\,\left(\int_{\R} \seps \, |x|^2 \diff x\right)^{1/2} + M_s.
\end{align*}
Dividing by $(1 + \int_{\R} \seps \, |x|^2 \diff x)^{1/2}$, we obtain
$$
2\, \partial_t \left(1 + \int_{\R} \seps \, |x|^2 \diff x\right)^{1/2} \leq 2\,\left(\int_{\R} \seps \, |\partial_x p(\seps)|^2 \diff x \right)^{1/2} + M_s^{1/2} + M_s. 
$$
Finally, integrating in time and using \eqref{eq:estimate_used_to_prove_the_bound_on_the_2nd_moment}, we obtain
$$
\left(1 + \int_{\R} \seps(T,x) \, |x|^2 \diff x\right)^{1/2} \leq C(\alpha, s^0)\, (1+T).
$$
The conclusion follows.
\end{proof}

\begin{lem}\label{lem:arbitrary_high_moments}
Let $\alpha>1$, $\ueps, \veps$ be the solution to \eqref{eq:general_cross_diffusion_intro_viscosity}. Let $k\geq 2$ be even. Then, there exists a constant $C=C(\alpha, \|s^0\|_{L^1(\R)}, \|s^0\|_{L^{\infty}(\R)},  \int_{\R} s^0\, |x|^k \diff x, T)$, independent of $\eps$, such that $$
\sup_{t\in[0,T]}\int_{\R} \seps(t,x)\, |x|^k\diff x \leq C.
$$
\end{lem}
\begin{proof}
Let $\phi_k(x) = (1+|x|^2)^{k/2}$. We compute 
$$
\partial_t \int_{\R} \seps \, \phi_k(x) \diff x = \frac{1}{\alpha} \int_{\R} \seps^{\alpha} \, |\partial^2_x \phi_k(x)| \diff x + \int_{\R} \seps\, |\partial_x \phi_k(x)| \diff x + \eps\int_{\R} \seps\,|\partial^2_x \phi_k(x)| \diff x.  
$$
We estimate $\seps^{\alpha} \leq \|\seps\|_{L^{\infty}}^{\alpha-1} \, \seps$, where the $L^{\infty}$ norm is finite by Lemma~\ref{lem:quantitative_dependence_constants}. Moreover, it is not difficult to see that $|\partial^2_x \phi_k| \leq C(k)\, \phi_{k-2}$ and $|\partial_x \phi_k| \leq C(k)\, \phi_{k-1}$. Since 
$$
\int_{\R} \seps\, |x|^{k-1} \diff x \leq \int_{\R} \seps\, |x|^{k-2} \diff x + \int_{\R} \seps\, |x|^{k} \diff x,
$$
the conclusion follows.
\end{proof}
\begin{lem}\label{eq:ulogu_lower_bound_in_terms_of_2nd_moment}
Let $u, v \geq 0$, $s = u+v$. Let $M_{u} = \int_{\R} u \diff x < \infty$, $I_u = \int_{\R} u\,|x|^2 \diff x < \infty$, and analogously for $v$. Then,
\begin{align*}
\int_{\R} &\left(u\log u + v\log v - \frac{1}{2} s\log s\right) \diff x \\ 
&\geq  -\frac{M_u}{4} \log I_u- \frac{\pi}{2} \, M_u^2 -\frac{M_v}{4} \log I_v- \frac{\pi}{2} \, M_v^2 - \frac{\log 2}{2} (M_u + M_v).
\end{align*}
\end{lem}
\begin{proof}
We first write 
$$
u\log u + v\log v - \frac{1}{2} s\log s = \frac{1}{2}(u\log u + v\log v) + \frac{1}{2}s \left(\frac{u}{s} \log \frac{u}{s} + \frac{v}{s} \log \frac{v}{s}\right).
$$
A calculus computation shows that the second term is bounded from below by $-s\,\log2 /2$. Hence, we only need to control the first term. We concentrate on $u\log u$. Let $g(x) = M_u(2\pi \sigma^2)^{-1/2}\,\exp(-|x|^2/2\sigma^2)$ be the Gaussian distribution with the same mass as $u$. By Jensen's inequality, $\int_{\R} u\log (u/g) \diff x  \geq 0$ so that $\int_{\R} u\log u \diff x \geq \int_{\R} u\log g \diff x$. We compute
$$
\int_{\R} u\log g \diff x = -M_u \log\big(M_u \sqrt{2\pi\sigma^2}\big) -\frac{I_u}{2\sigma^2}.
$$
Choosing $\sigma^2 = I_u/(2\pi M_u^2)$, we obtain
$$
\int_{\R} u\log u \diff x \geq -\frac{M_u}{2} \log I_u- \pi \, M_u^2.
$$
An analogous inequality holds for $v$ and the proof is concluded.
\end{proof}

\bibliographystyle{abbrv}
\bibliography{fastlimit}
\end{document}